\documentclass[12pt,a4paper,notitlepage,dvips]{article}
\usepackage[dvips]{graphicx}   
\usepackage{times}     
\usepackage[T1]{fontenc}  
\usepackage{graphicx,enumerate,amsmath,amsthm,mathrsfs,array,pstricks}
\usepackage{amsfonts}
\usepackage{amssymb}
\usepackage[a4paper,left=3cm,right=2cm,top=2.5cm,bottom=2.5cm]{geometry}
\usepackage{picins}

\newtheoremstyle{bthm}{\baselineskip}{\baselineskip}{\slshape}{}{\bfseries}{}{ }{}
\newtheoremstyle{bex}{\baselineskip}{\baselineskip}{}{}{\sffamily}{:}{\newline }{}
\theoremstyle{bthm}
\newtheorem{thm}{Theorem}[section]
\newtheorem{cor}[thm]{Corollary}

\newtheorem{prop}[thm]{Proposition}

\theoremstyle{bex}

\begin{document}
\begin{titlepage}
\title{On indicated coloring of some classes of graphs}
\author{P. Francis$^{1}$, S. Francis Raj$^{2}$ and M. Gokulnath$^3$}
\date{{\footnotesize Department of Mathematics, Pondicherry University, Puducherry-605014, India.}\\
{\footnotesize$^{1}$: selvafrancis@gmail.com\ $^{2}$: francisraj\_s@yahoo.com\ $^3$: gokulnath.math@gmail.com }}
\maketitle
\renewcommand{\baselinestretch}{1.3}\normalsize
\begin{abstract}
Indicated coloring is a type of game coloring in which two players collectively color the vertices of a
graph in the following way. In each round the first player (Ann) selects a vertex, and then
the second player (Ben) colors it properly, using a fixed set of colors. The goal of Ann is to
achieve a proper coloring of the whole graph, while Ben is trying to prevent the realization of
this project. The smallest number of colors necessary for Ann to win the game on a graph $G$
(regardless of Ben's strategy) is called the indicated chromatic number of $G$, denoted by
$\chi_i(G)$. In this paper, we obtain structural characterization of connected $\{P_5,K_4,Kite,Bull\}$-free graphs which contains
an induced $C_5$ and connected $\{P_6,C_5, K_{1,3}\}$-free graphs that contains an induced $C_6$. Also, we prove that $\{P_5,K_4,Kite,Bull\}$-free graphs that contains an induced $C_5$ and   $\{P_6,C_5,\overline{P_5}, K_{1,3}\}$-free graphs which contains an induced $C_6$ are $k$-indicated colorable for all $k\geq\chi(G)$. In addition,  we show that $\mathbb{K}[C_5]$ is $k$-indicated colorable for all $k\geq\chi(G)$ and as a consequence, we  exhibit that $\{P_2\cup P_3, C_4\}$-free graphs, $\{P_5,C_4\}$-free graphs are $k$-indicated colorable for all $k\geq\chi(G)$. This partially answers one of the questions which was raised by A. Grzesik in \cite{and}.
\end{abstract}

\section{Introduction}

All graphs considered in this paper are simple, finite and undirected. For any positive integer $k$, a proper $k$-coloring of a graph $G$ is a mapping $c$ : $V(G)\rightarrow\{1,2,\ldots,k\}$ such that for any two adjacent vertices $u,v\in V(G)$, $c(u)\neq c(v)$. A graph is said to be $k$-colorable if it admits a proper $k$-coloring. The chromatic number $\chi(G)$ of a graph $G$ is the smallest $k$ such that $G$ is $k$-colorable. In this paper, $P_n,C_n$ and $ K_n$ respectively denotes the path, the cycle and the complete graph on $n$ vertices. For $S,T\subseteq V(G)$, let $\langle S\rangle$ denote the subgraph induced by $S$ in $G$ and let $[S,T]$ denote the set of all edges with one end in $S$ and the other end in $T$. $ [S, T ]$ is said to be complete if every vertex in $S$ is adjacent with every vertex in $T$. For any graph $G$, let $\overline{G}$ denote the complement of $G$.

Let us recall some of the definitions which are required for this paper. Let $\mathcal{F}$ be a family of graphs. We say that $G$ is $\mathcal{F}$-free if it contains no induced subgraph which is isomorphic to a graph in $\mathcal{F}$. For two vertex-disjoint graphs $G_1$ and $G_2$, the join of $G_1$ and $G_2$, denoted by $G_1+G_2$, is the graph whose vertex set $V(G_1+G_2) = V(G_1)\cup V(G_2)$ and the edge set $E(G_1+G_2) = E(G_1)\cup E(G_2)\cup\{xy: x\in V(G_1),\ y\in V(G_2)\}$. In this paper, we write $H\sqsubseteq G$ if $H$ is an induced subgraph of $G$. Next, the coloring number of a graph $G$, denoted by col$(G)$, is defined by col$(G)=1+\max\limits_{H\subseteq G}\delta(H)$. By Szekeres-Wilf's inequality, $\chi(G)\leq \mathrm{col}(G)$.

A game coloring of a graph is a coloring of the vertices in which two players Ann and Ben are alternatively coloring the vertices of the graph $G$ properly by using a fixed set of colors $C$. The first player Ann is aiming to get a proper coloring of the whole graph, where as the second player Ben is trying to prevent the realization of this project. If all the vertices are colored then Ann wins the game, otherwise Ben wins (that is, at that stage of the game there appears a block vertex. A $block$ vertex means an uncolored vertex which has all colors from $C$ on its neighbors). The minimum number of colors required for Ann to win the game on a graph $G$ irrespective of Ben's strategy is called the game chromatic number of the graph $G$ and it is denoted by $\chi_g(G)$. There has been a lot of papers on  game coloring. See for instance, \cite{gua,sek,wu,zhu}. The idea of indicated coloring was introduced by A. Grzesik  in \cite{and} as a slight variant of the game coloring in the following way: in each round the first player Ann selects a vertex and then the second player Ben colors it properly, using a fixed set of colors. The aim of Ann as in game coloring is to achieve a proper coloring of the whole graph $G$, while Ben tries to ``block'' some vertex.
The smallest number of colors required for Ann to win the game on a graph $G$ is known as the indicated chromatic number of $G$ and is denoted by $\chi_i(G)$. Clearly from the definition we see that $\omega(G)\leq\chi(G)\leq\chi_i(G)\leq\Delta(G)+1$. For a graph $G$, if Ann has a winning strategy while using $k$ colors, then we say that $G$ is $k$-indicated colorable.


In \cite{zhu}, X. Zhu has asked the following question for game coloring. Whether increasing the number of colors will favor Ann? That is, if Ann has a winning strategy using $k$ colors, will Ann have a winning strategy using $k+1$ colors? The same question was asked by A. Grzesik for indicated coloring. The question can be equivalently stated as ``Whether $G$ is $k$-indicated colorable for every $k\geq\chi_i(G)$''.  He also showed by an example that the increase in number of colors does not make life simple for Ann rather it makes it much harder. There has been already some partial answers to this question. For instance in \cite{fran,pan}, R. Pandiya Raj et.al.  have shown that chordal graphs, cographs, complement of bipartite graphs, $\{P_5,K_3\}$-free graphs, $\{P_5,$paw$\}$-free graphs, and $\{P_5,K_4-e\}$-free graphs are $k$-indicated colorable for all $k\geq\chi(G)$.
In addition, M. Laso$\acute{\mathrm{n}}$ in \cite{las} has obtained the indicated chromatic number of matroids. In this paper, we obtain structural characterization of connected $\{P_5,K_4,Kite,Bull\}$-free graphs which contains an induced $C_5$ and connected $\{P_6,C_5, K_{1,3}\}$-free graphs which contains an induced $C_6$. Also, we prove that $\{P_5,K_4,Kite,Bull\}$-free graphs that contains an induced $C_5$ and   $\{P_6,C_5,\overline{P_5}, K_{1,3}\}$-free graphs which contains an induced $C_6$ are $k$-indicated colorable for all $k\geq\chi(G)$. In addition,  we show that $\mathbb{K}[C_5]$,  the complete expansion of $C_5$, is $k$-indicated colorable for all $k\geq\chi(G)$ and as a consequence, we exhibit that $\{P_2\cup P_3, C_4\}$-free graphs, $\{P_5,C_4\}$-free graphs are $k$-indicated colorable for all $k\geq\chi(G)$.

Notations and terminologies not mentioned here are as in \cite{west}.

\section{Structural characterization of some free graphs and their indicated coloring}

In \cite{boh}, it has been shown that the game chromatic number of  a bipartite graph can be arbitrarily large when compared to the chromatic number which is equal to 2. But while considering the indicated chromatic number of a bipartite graph $G$, A. Grzesik in \cite{and} has shown that $\chi_i(G)=2$.

\begin{thm}$\mathrm{(\cite{and}})$
\label{bi}Every bipartite graph is $k$-indicated colorable for every $k\geq2$.
\end{thm}

Next, let us recall the definition of complete expansion and independent expansion of a graph $G$. Let $G$ be a graph on $n$ vertices $v_1,v_2,\ldots,v_n$, and let $H_1,H_2,\ldots,H_n$ be $n$ vertex-disjoint graphs. An expansion
$G(H_1,H_2,\ldots,H_n)$ of $G$ is the graph obtained from $G$ by\\
(i)  replacing each $v_i$ of $G$ by $H_i$, $i=1,2,\ldots,n$, and\\
(ii) by joining every vertex in $H_i$ with every vertex in $ H_j$ whenever $v_i$ and $v_j$ are adjacent in $G$.

For $i\in\{1,2,\ldots,n\}$, if $H_i=K_{m_i}$, then $G(H_1,H_2,\ldots,H_n)$ is said to be a complete expansion of $G$ and is denoted by $\mathbb{K}[G](m_1,m_2,\ldots,m_n)$ or $\mathbb{K}[G]$. For $i\in\{1,2,\ldots,n\}$, if  $H_i=\overline{K_{m_i}}$, then $G(H_1,H_2,\ldots,H_n)$ is said to be an independent expansion of $G$ and is denoted by $\mathbb{I}[G](m_1,m_2,\ldots,m_n)$ or $\mathbb{I}[G]$.

In \cite {sum}, D. P. Sumner studied the structural property of $\{P_5,K_3\}$-free graphs.

\begin{thm}\label{p5k3}
$\mathrm{(\cite{sum}})$ Let $G$ be a $\{P_5,K_3\}$-free graph. Then each component of $G$ is either bipartite or $\mathbb{I}[C_5](m_1,m_2,\ldots,m_5)$, where $m_i\geq1$ for $ i=1,2,3,4,5$.
\end{thm}

Let us start this section with a structural characterization of a family of $P_5$-free graphs. The study of $P_5$-free graphs has been of interest for a lot of coloring parameters. For instance see, \cite{bac,bra2,fou}. In this direction, we would like to consider connected $\{P_5,K_4,Kite,Bull\}$-free graphs that contains an induced $C_5$. Here, the graphs Kite and Bull are shown in Figure \ref{free}.

\begin{figure}
   \centering
\psset{unit=.45in}
{
\begin{pspicture}(0,-1.18)(11.34,1.16)
\psdots[dotsize=0.16](0.86,1.08)
\psdots[dotsize=0.16](0.06,0.68)
\psdots[dotsize=0.16](1.66,0.68)
\psdots[dotsize=0.16](0.86,0.28)
\psdots[dotsize=0.16](0.06,0.68)
\psline[linewidth=0.03cm](0.86,1.08)(0.06,0.68)
\psline[linewidth=0.03cm](0.06,0.68)(0.86,0.28)
\psline[linewidth=0.03cm](0.86,0.28)(1.66,0.68)
\psline[linewidth=0.03cm](0.86,1.08)(1.66,0.68)
\psdots[dotsize=0.16](3.26,1.08)
\psdots[dotsize=0.16](2.46,0.68)
\psdots[dotsize=0.16](4.06,0.68)
\psdots[dotsize=0.16](3.26,0.28)
\psdots[dotsize=0.16](2.46,0.68)
\psline[linewidth=0.03cm](3.26,1.08)(2.46,0.68)
\psline[linewidth=0.03cm](2.46,0.68)(3.26,0.28)
\psline[linewidth=0.03cm](3.26,0.28)(4.06,0.68)
\psline[linewidth=0.03cm](3.26,1.08)(4.06,0.68)
\psdots[dotsize=0.16](5.66,1.08)
\psdots[dotsize=0.16](4.86,0.68)
\psdots[dotsize=0.16](6.46,0.68)
\psdots[dotsize=0.16](5.66,0.28)
\psdots[dotsize=0.16](4.86,0.68)
\psline[linewidth=0.03cm](5.66,1.08)(4.86,0.68)
\psline[linewidth=0.03cm](4.86,0.68)(5.66,0.28)
\psline[linewidth=0.03cm](5.66,0.28)(6.46,0.68)
\psline[linewidth=0.03cm](5.66,1.08)(6.46,0.68)
\psdots[dotsize=0.16](8.06,1.08)
\psdots[dotsize=0.16](8.86,0.48)
\psdots[dotsize=0.16](7.26,0.48)
\psdots[dotsize=0.16](10.46,1.08)
\psdots[dotsize=0.16](11.26,0.48)
\psdots[dotsize=0.16](9.66,0.48)
\psdots[dotsize=0.16](0.86,-0.52)
\psdots[dotsize=0.16](0.86,1.08)
\psline[linewidth=0.03cm](0.86,1.08)(0.86,0.28)
\psline[linewidth=0.03cm](0.06,0.68)(0.86,-0.52)
\psline[linewidth=0.03cm](1.66,0.68)(0.86,-0.52)
\psdots[dotsize=0.16](3.26,-0.52)
\psdots[dotsize=0.16](5.66,-0.52)
\psline[linewidth=0.03cm](3.26,0.28)(3.26,-0.52)
\psline[linewidth=0.03cm](5.66,0.28)(5.66,-0.52)
\psline[linewidth=0.03cm](3.26,1.08)(3.26,0.28)
\psline[linewidth=0.03cm](4.86,0.68)(6.46,0.68)
\psdots[dotsize=0.16](7.26,-0.52)
\psdots[dotsize=0.16](8.86,-0.52)
\psdots[dotsize=0.16](9.66,-0.52)
\psdots[dotsize=0.16](11.26,-0.52)
\psline[linewidth=0.03cm](7.26,0.48)(7.26,-0.52)
\psline[linewidth=0.03cm](8.86,0.48)(8.86,-0.52)
\psline[linewidth=0.03cm](7.26,0.48)(8.86,0.48)
\psline[linewidth=0.03cm](8.06,1.08)(7.26,0.48)
\psline[linewidth=0.03cm](8.06,1.08)(8.86,0.48)
\psline[linewidth=0.03cm](7.26,-0.52)(8.86,-0.52)
\psline[linewidth=0.03cm](9.66,0.48)(9.66,-0.52)
\psline[linewidth=0.03cm](11.26,0.48)(11.26,-0.52)
\psline[linewidth=0.03cm](9.66,0.48)(10.46,1.08)
\psline[linewidth=0.03cm](10.46,1.08)(11.26,0.48)
\psline[linewidth=0.03cm](11.26,0.48)(9.66,0.48)
\usefont{T1}{ptm}{m}{n}
\rput(0.83,-0.975){\footnotesize $(\overline{P_2\cup P_3})$}
\usefont{T1}{ptm}{m}{it}
\rput(3.2059374,-0.975){\footnotesize Dart}
\usefont{T1}{ptm}{m}{it}
\rput(5.6339064,-0.975){\footnotesize Kite}
\usefont{T1}{ptm}{m}{it}
\rput(8.1039065,-0.975){\footnotesize $\overline{P_5}$}
\usefont{T1}{ptm}{m}{it}
\rput(10.444843,-0.975){\footnotesize Bull}
\end{pspicture}
\caption{Some special graphs}
\label{free}
}
\end{figure}
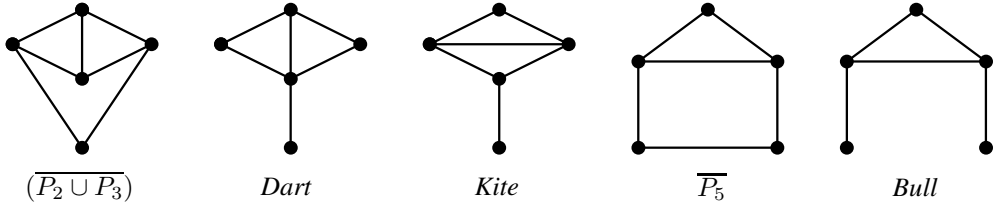

\begin{thm} \label{p5k4kb}
If $G$ is a connected $\{P_5,K_4,Kite,Bull\}$-free graph that contains an induced $C_5$, then $V(G)=V_1\cup V_2\cup V_3$ such that (1) $\langle V_2\rangle$ is a complete bipartite graph with bipartitions $B$ and $S$, (2) $\langle V_1\cup V_3\rangle$ is disjoint union of $\mathbb{I}[C_5]'s$ and bipartite graphs, (3) $[V_1,B]$ is complete, $[V_1,S]=[V_1,V_3]=[V_3,B]=\emptyset$ and (4) there exists $x^*\in S$ such that $[x^*,V_3]$ is complete.
\end{thm}
\begin{proof}
Let $G$ be a connected $\{P_5,K_4,Kite,Bull\}$-free graph that contains an induced $C_5\cong \langle\{v_0,v_1,v_2,v_3,v_4\}\rangle$ $=\langle N_0\rangle$, and let $N_i=\{x\in V(G)
:\mathrm{ dist}(x,N_0)=i\},$ $i\geq1$.

\noindent \textbf{Claim 1:} If $x\in N_1$, then $\langle N(x)\cap N_0\rangle\cong 2K_1$ or $C_5$.

For $x\in N_1$, the possibilities for $\langle N(x)\cap N_0\rangle$ are $K_1, K_2, P_3, P_4, 2K_1, K_1\cup K_2$ and $C_5$. Here (a) if $\langle N(x)\cap N_0\rangle\cong K_1$ or $K_2$, then $P_5 \sqsubseteq G$, (b) if $\langle N(x)\cap N_0\rangle\cong P_3$ or $P_4$, then Kite $\sqsubseteq G$, and (c) if $\langle N(x)\cap N_0\rangle\cong K_1\cup K_2$, then Bull $\sqsubseteq G$, a contradiction. Finally, if $\langle N(x)\cap N_0\rangle\cong 2K_1$ or $C_5$, we cannot get $P_5$ or $K_4$ (or) Kite (or) Bull as an induced subgraph in $\langle N_0\cup N_1\rangle$. Hence $\langle N(x)\cap N_0\rangle\cong 2K_1$ or $C_5$.

Throughout this proof, for any integer $i$, $v_i$ means $v_{i\ (\textnormal{mod}\ 5)}$ and $A_i$ means $A_{i\ (\textnormal{mod}\ 5)}$. For $0\leq i\leq 4$, let $A_i=\{x\in N_1:  N(x)\cap N_0=\{v_{i-1},v_{i+1}\}\}\cup \{v_i\}$ and let $B=\{x\in N_1:\langle N(x)\cap N_0\rangle\cong C_5\}$.

\noindent \textbf{Claim 2:} $\langle\cup_{i=0}^4A_i\rangle\cong \mathbb{I}[C_5]$.

For every $i$, $0\leq i\leq 4$, we have (a) $\langle A_i\rangle$ is independent (else, if $x,y\in A_i$ are adjacent, then $x,y\neq v_i$ and $\langle\{ x,v_{i+1},y,v_{i-1},v_{i-2}\}\rangle\cong\ \mathrm{Kite}\sqsubseteq G$), (b) $[ A_i,A_{i+1}]$ is complete (else, if $x\in A_i$ and $y\in A_{i+1}$ are not adjacent, then $\langle\{ x, v_{i-1},v_{i-2},v_{i-3},y\}\rangle\cong P_5\sqsubseteq G$), (c) $[ A_i, A_{i+2}]=\emptyset$ (else, if $x\in A_i$ and $y\in A_{i+2}$ are adjacent, then $\langle \{v_{i-1},x, v_{i+1},v_{i+2},y\}\rangle\cong\ \mathrm{Bull} \sqsubseteq G$). Thus from (a), (b) and (c), we conclude that $\langle\cup_{i=0}^4A_i\rangle\cong \mathbb{I}[C_5]$.

\noindent \textbf{Claim 3:} $[\cup_{i=0}^4A_i,B]$ is complete.

On the contrary, if there exist vertices $x\in A_i$, $x\neq v_i$ and $y\in B$ such that $xy\notin E(G)$, then $\langle\{v_i,v_{i+1},x,y,v_{i-2}\}\rangle\cong\ \mathrm{Bull} \sqsubseteq G$, a contradiction.

\noindent \textbf{Claim 4:} $\langle B\rangle$ is independent.

Suppose if there exist vertices $x$ and $y$ in $B$ such that $xy\in E(G)$, then $\langle \{v_1,v_2,x,y\} \rangle\cong K_4\sqsubseteq G$, a contradiction.

\noindent \textbf{Claim 5:} If $x\in \cup_{i=0}^4A_i$, then $N(x)\cap N_2=\emptyset$.

Let $x\in A_i$ for some $i$ such that $0\leq i\leq4$. Suppose if there exists a vertex $y\in N(x)\cap N_2$, then $\langle \{y,x,v_{i+1},v_{i+2},v_{i+3}\}\rangle\cong P_5\sqsubseteq G$, a contradiction.\\
Note that, if $B=\emptyset$, then by using claims 2 and 5, we can observe that $G\cong \mathbb{I}[C_5]$. Now, let us assume that $B\neq \emptyset $ and $N_2\neq \emptyset$.

\noindent \textbf{Claim 6:} $[B,N_2]$ is complete.

Here, if there exist vertices $x\in B$ and $y\in N_2$ such that $xy\notin E(G)$,  then by using Claim 5, there exists a vertex $z\in B$ such that $yz\in E(G)$. Now from Claim 4, $xz\notin E(G)$. Hence $\langle\{v_1,v_2,x,y,z\}\rangle\cong\ \mathrm{Kite}\sqsubseteq G$, a contradiction.

\noindent \textbf{Claim 7:} $\langle N_2\rangle$ is triangle-free.

On the contrary, assume that there exist vertices $\{u_1,u_2,u_3\}\subseteq N_2$ which induce a $K_3$ in $G$. Then by using Claim 6, for every vertex $x\in B$, $\langle\{x, u_1,u_2,u_3\}\rangle\cong K_4\sqsubseteq G$, a contradiction.\\
Since $G$ is assumed to be $P_5$-free and $\langle N_2\rangle$ is triangle-free, by Theorem \ref{p5k3} we see that each component of $\langle N_2\rangle$ is either isomorphic to a $\mathbb{I}[C_5]$ or to a bipartite graph.

Suppose $N_3=\emptyset$, by using the above Claims, we see that $G\cong \langle \cup_{i=0}^4 A_i\cup N_2\rangle + \langle B\rangle$. Now, let us assume that $N_3\neq \emptyset$.

\noindent \textbf{Claim 8:} If $xy$ is an edge in $\langle N_2\rangle$, then $N(x)\cap N_3=\emptyset$ and $N(y)\cap N_3=\emptyset$.

Let $xy$ be an edge in $\langle N_2\rangle$.  Suppose if there exists a vertex $z\in N_3$ such that $xz\in E(G)$ or $yz\in E(G)$ (or) $\{xz,yz\}\in E(G)$,  then $\langle\{v_1,b, x,y,z\}\rangle\cong \ \mathrm{Bull} \mathrm{~or~} \mathrm{Kite}\sqsubseteq G$ (where $b\in B$), a contradiction.

Let  $S$ be the collection of the vertices in $N_2$ which have neighbors in $N_3$. From Claim 8, it can be seen that $S$ is an independent subset of $N_2$ such that $[S,N_2\backslash S]=\emptyset$.

\noindent \textbf{Claim 9:} There exists a vertex $x^*\in S$ such that $[x^*,N_3]$ is complete. Also, $\langle N_3\rangle$ is triangle-free.

On the contrary, let us assume that there exists no $x^*\in S$ such that  $[x^*,N_3]$ is complete. Under this assumption, first let us show that there exist vertices $x,x'\in S$ and $y,y'\in N_3$ such that $xy, x'y'\in E(G)$ and $xy',x'y\notin E(G)$. Let $x_1,x_2,\ldots, x_{|S|}$ and $y_1,y_2,\ldots, y_{|N_3|}$ be the vertices of $S$ and $N_3$ respectively. Consider $x_1$. By our assumption, $x_1$ is non-adjacent to at least one of the vertex in $N_3$, say $y_1$ and the vertex $y_1$ should have a neighbor in $S$, say $x_2$. Now the vertex $x_2$ is also non-adjacent to at least one vertex in $N_3$, say $y_2$. Suppose $x_1y_2\in E(G)$, then  $x=x_1$, $x'=x_2$, $y=y_2$ and $y'=y_1$ will possess the required property. If not, $x_1y_2\notin E(G)$ and $y_2$ should have a neighbor in $S$, say $x_3$. Suppose $x_3y_1\notin E(G)$,  $x=x_2$, $x'=x_3$, $y=y_1$ and $y'=y_2$ will have the required property. Otherwise, $x_3y_1\in E(G)$ and there is a vertex in $N_3$ which is non-adjacent to $x_3$, say $y_3$. Like wise, if $x_1y_3\in E(G)$ or $x_2y_3\in E(G)$, then as mentioned above we can get vertices with the required condition. If not, $x_1y_3\notin E(G)$ and $x_2y_3\notin E(G)$, and hence the vertex $y_3$ should have a neighbor in $S$, say $x_4$. Similarly, even when the vertex $x_4$ is non-adjacent to $y_1$ or $y_2$, we can get the vertices with the required condition. Suppose $x_4$ is adjacent to $y_1$ and $y_2$, the process continues. Since the number of vertices is finite, this process stops at a certain stage having vertices $x_i,x_j\in S$ and $y_{i-1}, y_i\in N_3$  such that $x_iy_{i-1},y_ix_j\in E(G)$ and $x_jy_{i-1},x_iy_i\notin E(G)$ for some $i,j\in\{1,2,\ldots,|S|\}$. Thus there exist vertices $x,x'\in S$ and $y,y'\in N_3$ such that $xy, x'y'\in E(G)$ and $xy',x'y\notin E(G)$. Now by using Claim 8, $xx'\notin E(G)$, and hence for some $b\in B$, $\langle\{v_1,b,x',y',y\}\rangle\cong P_5$ when $yy'\in E(G)$ or $\langle\{y,x,b,x',y'\}\rangle\cong P_5$ when $yy'\notin E(G)$, a contradiction.

Also note that, $\langle N_3\rangle$ has to be  triangle-free.  Otherwise,  $K_4\sqsubseteq G$.

\noindent \textbf{Claim 10:} $N_i=\emptyset$, for all $i\geq4$.

This can be easily observed from the fact that if $N_4\neq\emptyset$, then we will get $P_5\sqsubseteq G$, a contradiction.

Since $G$ is $P_5$-free and $\langle N_3\rangle$ is triangle-free, each component of $\langle N_3\rangle$ is isomorphic to a $\mathbb{I}[C_5]$ or to a bipartite graph. Let $V_1=\{\cup_{i=0}^4 A_i\cup (N_2\setminus S)\}$, $V_2=B\cup S$ and $V_3=N_3$. By using the above Claims, we see that $\langle V_2\rangle$ is a complete bipartite graph, $\langle V_1\rangle$ is a disjoint union of $\mathbb{I}[C_5]'s$ and bipartite graphs such that $[V_1,B]$  is complete,  and $\langle V_3\rangle $ is also a disjoint union of $\mathbb{I}[C_5]'s$ and bipartite graphs such that there exists a vertex $x^*\in S$ such that $[x^*,V_3]$ is complete. Also from Claims 5 and 8, it can be observed that $[V_1,S]=[V_1,V_3]=[V_3,B]=\emptyset$.
\end{proof}
By Theorem \ref{p5k4kb}, one can easily find the chromatic number of this family.
\begin{cor} \label{p5k4kbc}
If $G$ is a connected $\{P_5,K_4,Kite,{Bull}\}$-free graph that contains an induced $C_5$, then $\chi(G)=3$ if and only if $G\cong \mathbb{I}[C_5]$, otherwise $\chi(G)=4$.
\end{cor}
\begin{proof}
By Theorem \ref{p5k4kb}, we see that $V(G)= V_1\cup V_2\cup V_3$, where $V_1, V_2$ and $V_3$ have the properties stated in the statement of Theorem \ref{p5k4kb}. Since $\langle V_1\cup V_3\rangle$ is a disjoint union of $\mathbb{I}[C_5]'s$ and bipartite graphs, one can color the vertices of $V_1$ and $V_3$ with colors $\{1,2,3\}$ and $\{2,3,4\}$ respectively which yields a proper coloring for the subgraph $\langle V_1\cup V_3\rangle$. Since $\langle V_2\rangle$ is a complete bipartite graph with bipartition $B$ and $S$, $[B,V_1]$ is complete, $[x^*,V_3]$ is complete and   $[V_1,S]=[V_1,V_3]=[V_3,B]=\emptyset$, coloring the vertices of $B$ and $S$ with 4 and 1 respectively will yields a proper coloring for $G$. Thus $\chi(G)\leq 4$. Suppose $B=\emptyset$, then $G\cong \mathbb{I}[C_5]$ and hence $\chi(G)=3$.
If not, $B\neq \emptyset$ and $\langle B\rangle+\mathbb{I}[C_5]\sqsubseteq G$. Thus $\chi(G)=4$.
\end{proof}

Now we shall consider the indicated coloring for the independent expansion of $C_n$.

\begin{thm}
\label{indexp} For  $1\leq i\leq n$, let $m_i$'s be  positive integers. Then the graph $G=$ \linebreak $\mathbb{I}[C_n](m_1,m_2,\ldots,m_n)$ is $k$-indicated colorable for all $k\geq\chi(G)$.
\end{thm}
\begin{proof}
For the graph $G = \mathbb{I}[C_n](m_1,m_2,\ldots,m_n)$, for $m_i\geq1$, $1\leq i\leq n$, it is easy to observe that $\chi(G)=2$ when $n$ is even and $\chi(G)=3$ when $n$ is odd.  Here,  if Ann first present the vertices of an induced $C_n$ cyclically and then the remaining vertices in any order, Ben will not be able to produce a blocked vertex. Thus Ann has a winning strategy for $G$ with $k$ colors, for every $k\geq \chi(G)$.
\end{proof}

We know that, for the union of two graphs $G_1$ and $G_2$, $\chi(G_1\cup G_2)=\max\{\chi(G_1),\chi(G_2)\}$. The same holds even for the indicated chromatic number.
\begin{thm}
$\mathrm{(\cite{pan})}$. \label{union}Let $G = G_1 \cup  G_2$. If $G_1$ is $k_1$-indicated colorable for every $k_1 \geq\chi_i(G_1)$ and $G_2$ is
$k_2$-indicated colorable for every $k_2\geq\chi_i(G_2)$, then $\chi_i(G) = \max\{\chi_i(G_1),\chi_i(G_2)\}$ and $G$ is $k$-indicated
colorable for all $k\geq\chi_i(G)$.
\end{thm}

\noindent Next, we see that Corollary \ref{p5k3c} which was proved in \cite{pan} is a simple consequence of Theorem \ref{bi},  \ref{p5k3},  \ref{indexp} and \ref{union}.
\begin{cor}$\mathrm{(\cite{pan})}$
\label{p5k3c} Every $\{P_5,K_3\}$-free graph $G$ is $k$-indicated colorable for all $k\geq\chi(G)$.
\end{cor}

Now, let us consider the indicated coloring of $\{P_5,K_4, Kite, Bull\}$-free graphs which contains an induced $C_5$.
\begin{thm}
Let $G$ be a $\{P_5,K_4,Kite,Bull\}$-free graph which contains an induced $C_5$. Then $G$ is $k$-indicated colorable for all $k\geq\chi(G)$.
\end{thm}
\begin{proof}
By Theorem \ref{union}, it is enough to prove the result for a connected $\{P_5,K_4,Kite,Bull\}$-free graph that contains an induced $C_5$.
Let $G$ be such a graph. Then by Theorem \ref{p5k4kb}, $V(G)=V_1\cup V_2\cup V_3$ where (1) $\langle V_2\rangle$ is a complete bipartite graph with bipartition say $B$ and $S$, (2) $\langle V_1\cup V_3\rangle$ is a disjoint union of $\mathbb{I}[C_5]'s$ and bipartite graphs, (3) $[V_1,B]$ is complete, $[V_1,S]=\emptyset$, $[V_1,V_3]=\emptyset$, $[V_3,B]=\emptyset$ and (4) there exists a vertex, say $x^*\in S$ such that $[x^*, V_3]$ is complete.

Suppose $B=\emptyset$, then $G\cong\mathbb{I}[C_5]$. Thus by Theorem \ref{indexp}, $G$ is $k$-indicated colorable for all $k\geq\chi(G)$. If not, $B\neq \emptyset$ and hence by Corollary \ref{p5k4kbc}, $G\ncong\mathbb{I}[C_5]$ and $\chi(G)=4$. Let $\{1,2,\ldots, k\geq4\}$ be the set of colors. We shall show that $G$ is $k$-indicated colorable. Let Ann start by presenting $x^*$ and a vertex $b\in B$. Without loss of generality, let the color used by Ben for $b$ and $x^*$ be 1 and 2 respectively. Since $[b,V_1]$ is complete and $[x^*, V_3]$ is complete, the set of available colors for $V_1$ and $V_2$ are  $\{2,3,\ldots,k\}$ and $\{1,3,4,\ldots,k\}$ respectively. Since $[V_1,V_3]=\emptyset$, $\langle V_1\cup V_3\rangle$ is a disjoint union of $\mathbb{I}[C_5]'s$ and bipartite graphs, by Theorem \ref{bi} and Theorem \ref{indexp}, $\langle V_1\cup V_3\rangle$ is $l$-indicated colorable for all $l\geq3$. That is,  Ann has a winning strategy for $\langle V_1\rangle$ while using the colors $\{2,3,\ldots,k\}$ and a winning strategy for $\langle V_3\rangle$ while using the colors $\{1,3,4,\ldots,k\}$. After presenting the vertices of $V_1$ and $V_3$ by using these winning strategies, Ann will present the remaining vertices of $B$ and $S$ in any order.  Clearly, the color of the vertices $b$ and $x^*$, namely 1 and 2 are available for the uncolored vertices of $B$ and $S$ respectively. Thus Ann wins the game on $G$ with $k$ colors,  $k\geq4$.
\end{proof}

\noindent Next, let us consider a structural characterization of a family of $P_6$-free graphs. The study of $P_6$-free graphs has  also been of interest for a lot of coloring parameters. See for instance, \cite{hof,liu,ran}. Here, we would like to consider  connected $\{P_6,C_5, K_{1,3}\}$-free graphs that contains an induced $C_6$.

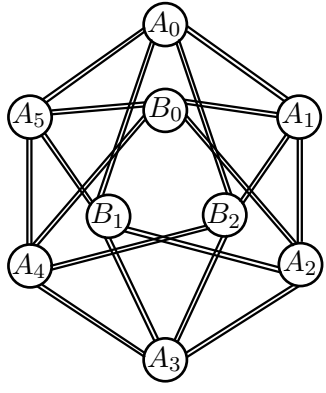
\begin{figure}
  \centering
\psset{unit=.35in}
{
\begin{pspicture}(0,-2.84)(4.68,2.84)
\pscircle[linewidth=0.04,dimen=outer](2.34,2.5){0.34}
\pscircle[linewidth=0.04,dimen=outer](4.32,1.12){0.34}
\pscircle[linewidth=0.04,dimen=outer](4.34,-1.08){0.34}
\pscircle[linewidth=0.04,dimen=outer](2.34,-2.5){0.34}
\pscircle[linewidth=0.04,dimen=outer](0.34,-1.1){0.34}
\pscircle[linewidth=0.04,dimen=outer](0.34,1.12){0.34}
\pscircle[linewidth=0.04,dimen=outer](2.34,1.22){0.34}
\pscircle[linewidth=0.04,dimen=outer](1.5,-0.36){0.34}
\pscircle[linewidth=0.04,dimen=outer](3.22,-0.34){0.34}
\psline[linewidth=0.03cm](0.52,1.38)(2.04,2.46)
\psline[linewidth=0.03cm](0.58,1.34)(2.04,2.4)
\psline[linewidth=0.03cm](0.3,0.82)(0.3,-0.78)
\psline[linewidth=0.03cm](0.36,0.8)(0.36,-0.8)
\psline[linewidth=0.03cm](0.44,-1.4)(2.04,-2.46)
\psline[linewidth=0.03cm](0.52,-1.38)(2.04,-2.4)
\psline[linewidth=0.03cm](2.64,-2.4)(4.24,-1.38)
\psline[linewidth=0.03cm](2.66,-2.46)(4.32,-1.4)
\psline[linewidth=0.03cm](4.36,0.82)(4.36,-0.74)
\psline[linewidth=0.03cm](4.3,0.8)(4.3,-0.76)
\psline[linewidth=0.03cm](2.66,2.48)(4.16,1.4)
\psline[linewidth=0.03cm](2.66,2.4)(4.1,1.36)
\psline[linewidth=0.03cm](2.04,1.34)(0.66,1.22)
\psline[linewidth=0.03cm](2.04,1.28)(0.64,1.16)
\psline[linewidth=0.03cm](2.62,1.38)(4.0,1.18)
\psline[linewidth=0.03cm](2.64,1.32)(4.0,1.12)
\psline[linewidth=0.03cm](0.5,0.86)(1.24,-0.2)
\psline[linewidth=0.03cm](0.56,0.88)(1.26,-0.14)
\psline[linewidth=0.03cm](1.44,-0.66)(2.18,-2.22)
\psline[linewidth=0.03cm](1.52,-0.68)(2.24,-2.2)
\psline[linewidth=0.03cm](4.18,0.84)(3.48,-0.18)
\psline[linewidth=0.03cm](4.14,0.88)(3.44,-0.14)
\psline[linewidth=0.03cm](3.26,-0.64)(2.52,-2.24)
\psline[linewidth=0.03cm](3.18,-0.66)(2.48,-2.2)
\psline[linewidth=0.03cm](2.02,1.14)(0.38,-0.78)
\psline[linewidth=0.03cm](2.04,1.08)(0.44,-0.8)
\psline[linewidth=0.03cm](2.64,1.16)(4.28,-0.76)
\psline[linewidth=0.03cm](2.62,1.1)(4.24,-0.8)
\psline[linewidth=0.03cm](1.42,-0.04)(2.18,2.24)
\psline[linewidth=0.03cm](1.34,-0.08)(2.14,2.28)
\psline[linewidth=0.03cm](2.5,2.24)(3.26,-0.04)
\psline[linewidth=0.03cm](2.54,2.28)(3.32,-0.04)
\psline[linewidth=0.03cm](1.76,-0.52)(4.02,-1.12)
\psline[linewidth=0.03cm](1.72,-0.58)(4.06,-1.2)
\psline[linewidth=0.03cm](2.94,-0.5)(0.64,-1.1)
\psline[linewidth=0.03cm](3.0,-0.54)(0.64,-1.16)
\usefont{T1}{ptm}{m}{n}
\rput(2.326875,2.51){\small $A_0$}
\usefont{T1}{ptm}{m}{n}
\rput(4.3185935,1.15){\small $A_1$}
\usefont{T1}{ptm}{m}{n}
\rput(4.3376563,-1.06){\small $A_2$}
\usefont{T1}{ptm}{m}{n}
\rput(2.3309374,-2.49){\small $A_3$}
\usefont{T1}{ptm}{m}{n}
\rput(0.32553125,-1.1){\small $A_4$}
\usefont{T1}{ptm}{m}{n}
\rput(0.3253125,1.15){\small $A_5$}
\usefont{T1}{ptm}{m}{n}
\rput(2.3453124,1.24){\small $B_0$}
\usefont{T1}{ptm}{m}{n}
\rput(1.48875,-0.35){\small $B_1$}
\usefont{T1}{ptm}{m}{n}
\rput(3.2051561,-0.35){\small $B_2$}
\end{pspicture}
\caption{$\{P_6,C_5,K_{1,3}\}$-free graph contains an induced $C_6$} 
\label{p6c5k13}
}
\end{figure}

\begin{thm}\label{p6c5}
If $G$ is a connected $\{P_6,C_5,K_{1,3}\}$-free graph which contains an induced $C_6$ then $G$ is isomorphic to the graph given in Figure \ref{p6c5k13}. Here $V(G)=(\cup_{i=0}^5A_i)\cup (\cup_{j=0}^2B_j)$ and the circle denote the complete subgraph induced by the sets $ A_i $ and $ B_j $ and the double line between any two sets denote the join of the two sets.
\end{thm}
\begin{proof}
Let $G$ be a connected $\{P_6,C_5,K_{1,3}\}$-free graph that contains an induced $C_6\cong $\newline $ \langle\{v_0,v_1,v_2,v_3,v_4,v_5\}\rangle=\langle N_0\rangle$, and let $N_i=\{x\in V(G):\mathrm{ dist}(x,N_0)=i\},$ $i\geq1$.

\noindent\textbf{Claim 1:} If $x\in N_1$, then $\langle N(x)\cap N_0\rangle\cong P_3$ or $2K_2$.

For $x\in N_1$, the possibilities for $\langle N(x)\cap N_0\rangle$ are $K_1, K_2, P_3, P_4,P_5, 2K_1,3K_1, 2K_2, K_1\cup K_2,K_1\cup P_3$ and $C_6$. Here (a) if $\langle N(x)\cap N_0\rangle \cong K_1$ or $K_2$, then $P_6\sqsubseteq G$,
(b) if $\langle N(x)\cap N_0\rangle\cong P_4$ or $K_1\cup K_2$, then $C_5\sqsubseteq G$, (c) if $\langle N(x)\cap N_0\rangle\cong P_5$ or $C_6$ (or) $K_1\cup P_3$ (or) $3K_1$ (or) $2K_1$, then $K_{1,3}\sqsubseteq G$, a contradiction. Finally,  if $\langle N(x)\cap N_0\rangle\cong P_3$ or $2K_2$, we see that neither $P_6$ nor $C_5$ (nor) $K_{1,3}$ is an induced subgraph of $\langle N_0\cup N_1\rangle$. Thus $\langle N(x)\cap N_0\rangle\cong P_3$ or $2K_2$.

Throughout this proof, for any integer $i$, $v_i$ means $v_{i\ (\textnormal{mod}\ 6)}$ and $A_i$ means $A_{i\ (\textnormal{mod}\ 6)}$. For $0 \leq i\leq 5$, let $A_i=\{x\in N_1:  N(x)\cap N_0=\{v_{i-1},v_i,v_{i+1}\}\}\cup \{v_i\}$ and  $B_i=\{x\in N_1:  N(x)\cap N_0=\{v_{i-2},v_{i-1},v_{i+1},v_{i+2}\}\}$.

\noindent\textbf{Claim 2:} $\langle\cup_{i=0}^5A_i\rangle\cong \mathbb{K}[C_6]$.

For every $i,$ $0 \leq i\leq 5$, we have (a) $\langle A_i\rangle$ is complete (suppose if there exist vertices $x,y \in A_i$ such that $xy\notin E(G)$,  then $\langle\{v_{i+1},v_{i+2},x,y\}\rangle\cong K_{1,3}\sqsubseteq G$), (b) $[A_i,A_{i+1}]$ is complete, (if not, there exist vertices $x \in A_i$ and $ y \in A_{i+1}$ such that $xy\notin E(G)$, and hence $\langle\{x,v_{i},y,v_{i+2},v_{i+3},v_{i+4}\}\rangle\cong P_6\sqsubseteq G$), (c) $[A_i,A_{i+2}]=\emptyset$, (suppose if there exist vertices $x \in A_i$ and $ y \in A_{i+2}$ such that $xy\in E(G)$, then $\langle\{x,y,v_{i+3},v_{i+4},v_{i+5}\}\rangle\cong C_5\sqsubseteq G$), (d) $[A_i,A_{i+3}]=\emptyset$, (otherwise as shown previously, we can find $x \in A_i$ and $y \in A_{i+3}$ such that $xy\in E(G)$, and $\langle\{x,v_{i-1},v_{i+1},y\}\rangle\cong K_{1,3}\sqsubseteq G$). Thus from (a), (b), (c) and (d), it can be seen that $\langle\cup_{i=0}^5A_i\rangle\cong \mathbb{K}[C_6]$.

\noindent\textbf{Claim 3:} $\langle B_i\rangle$ is complete, for $i=0,1,2,3,4,5$.

Here, if there exist vertices $x,y\in B_i$ such that $xy\notin E(G)$, then $\langle\{v_{i-1},v_{i},x,y\}\rangle\cong K_{1,3}\sqsubseteq G$, a contradiction.

\noindent\textbf{Claim 4:} $[B_i,B_{i+1}]=\emptyset$, for $i=0,1,2,3,4,5$.

Suppose if there exist vertices $x \in B_i$ and $y \in B_{i+1}$ such that $xy\in E(G)$, then $\langle\{x,y,v_{i+1},v_{i-2}\}\rangle\cong K_{1,3}\sqsubseteq G$, a contradiction.

\noindent\textbf{Claim 5:} $[A_i,B_{i}]=\emptyset$, $i=0,1,2,3,4,5$.

On the contrary, if there exist vertices $x \in A_i$ and $y \in B_{i}$ such that $xy\in E(G)$, then $\langle\{y,x,v_{i-2},v_{i+2}\}\rangle\cong K_{1,3}\sqsubseteq G$, a contradiction.

\noindent\textbf{Claim 6:} $[A_i,B_{i+1}]$ is complete, for $i=0,1,2,3,4,5$.

If not, there exist vertices $x \in A_i$ and $y \in B_{i+1}$ such that $xy\notin E(G)$. Here $\langle\{x,v_{i-1},y,v_{i+2},v_{i+1}\}\rangle\cong C_5\sqsubseteq G$, a contradiction.

\noindent\textbf{Claim 7:} $[A_i,B_{i+2}]$ is complete, for $i=0,1,2,3,4,5$.

It is easy to observe that if there exist vertices $x \in A_i$ and $y \in B_{i+2}$ such that $xy\notin E(G)$, then $\langle\{x,v_{i-1},v_{i-2},y,v_{i+1}\}\rangle\cong C_5\sqsubseteq G$, a contradiction.

\noindent\textbf{Claim 8:} $N_i=\emptyset$, for all $i$, $i\geq2$.

It is enough to show that $N_2=\emptyset$. Suppose $N_2\neq\emptyset$, then there exists a vertex $x\in N_2$. Since $G$ is connected, there exists a vertex $y\in A_j$ or $y\in B_j$ for some $j\in\{0,1,\ldots,5\}$ such that $xy\in E(G)$. Then $\langle\{y,v_{j-1},v_{j+1},x\}\rangle \cong K_{1,3}\sqsubseteq G$, a contradiction. Thus $V(G)=N_0\cup N_1$.

Note that $B_j=B_{j+3}$ for every $j\in\{0,1,2\}$. From all these Claims, we see that $G$ will be isomorphic to the graph shown in Figure \ref{p6c5k13}.
\end{proof}

An immediate consequence of Theorem \ref{p6c5} is given in Corollary \ref{p6p5c}.
\begin{cor}\label{p6p5c}
If $G$ is a connected $\{P_6,C_5, \overline{P_5}, K_{1,3}\}$-free graph that contains an induced $C_6$ then  $G\cong \mathbb{K}[C_6]$.
\end{cor}
\begin{proof}
It can seen from Theorem \ref{p6c5} that it is enough to show that $B_i=\{x\in N_1: N(x)\cap N_0=\{v_{i-2},v_{i-1},v_{i+1},v_{i+2}\}\}=\emptyset$, for every $i=0,1,2$. Suppose $x\in B_i$ for some  $i\in\{0,1,2\}$, then $\langle \{x,v_{i-1},v_i,v_{i+1},v_{i+2}\}\rangle\cong \overline{P_5}\sqsubseteq G$, a contradiction. Thus $G$ is isomorphic to a complete expansion of $C_6$.
\end{proof}

\noindent Even though the graph $G$ shown in Figure \ref{p6c5k13} looks simple, it looks challenging to obtain the indicated chromatic number of $G$. So, we have considered the indicated coloring of $ \mathbb{K}[C_6]$.
\begin{prop}\label{expc6}
For $1\leq i\leq 6$, let $m_i$'s be  positive integers. Then the graph $G=$\linebreak  $\mathbb{K}[C_6](m_1,m_2,m_3,m_4,m_5,m_6)$ is $k$-indicated colorable for all $k\geq\chi(G)$.
\end{prop}
\begin{proof}
Let us consider the graph $G=\mathbb{K}[C_6](m_1,m_2,m_3,m_4,m_5,m_6)$, where $m_i\geq1$ and $V_i=V(K_{m_i})$ for $1\leq i\leq 6$. Let $k$ be a positive integer such that $k\geq\chi(G)$ and let  $\{1,2,\ldots,k\}$ be the set of colors. We shall show that $G$ is $k$-indicated colorable. It is easy to see that $\chi(G)=\omega(G)$. Let Ann start by presenting the vertices of a maximum clique. Without loss of generality, let it be $V_1\cup V_2 $. Since $k\geq\omega(G)$, Ben has an available color for all the vertices in $V_1\cup V_2$ and let the colors given to $V_1$ and $V_2$ be $\{1,2,\ldots,m_1\}$ and $\{m_1+1,m_1+2,\ldots,m_1+m_2=\omega(G)\}$ respectively. Now, let Ann present the vertices of $V_3$ and $V_6$ (in any order). Since $V_1\cup V_2$ is maximum clique, $|V_3|\leq |V_1|$ and $|V_6|\leq |V_2|$, and hence Ben
has an available colors for all the vertices of $V_3$ and $V_6$. Since $[V_i,V_{i+1}]$ is complete, $\{1,2,\ldots,m_1\}$ and $\{m_1+1,m_1+2,\ldots,\omega(G)\}$ are colors available for $V_5$ and $V_4$ respectively. Now, let Ann present the vertices of $V_4$ until either the number of available colors for $V_5$ is equal to $|V_5|$ or every vertex in $V_4$ is colored. In either case, let Ann proceed by presenting all the vertices of $V_5$. If there are some uncolored vertices in $V_4$, let Ann present those vertices finally. Here it can be seen that, the number of available color for the vertices of $V_4\cup V_5$ is at least $\omega(G)$ and $|V_4|+|V_5|\leq \omega(G)$. Thus in this ordering, Ben will always have an available color for the remaining vertices.
\end{proof}

\noindent An immediate consequence of Theorem \ref{union}, Corollary \ref{p6p5c} and Proposition \ref{expc6} is given in Corollary \ref{p6c5ind}.
\begin{cor}\label{p6c5ind}
If $G$ is a $\{P_6,C_5, \overline{P_5}, K_{1,3}\}$-free graph that contains an induced $C_6$, then $G$ is $k$-indicated colorable for all $k\geq \chi(G)$.
\end{cor}

\section{Indicated coloring of $\mathbb{K}[C_5]$ and some of its consequences}

Let us start this Section by recalling two of the results which were proved in \cite{pan}.
\begin{thm}$\mathrm{(\cite{pan})}$.
\label{col}Any graph $G$ is $k$-indicated colorable for all $k\geq\mathrm{col}(G)$.
\end{thm}

We know that, for the join of two graphs $G_1$ and $G_2$, $\chi(G_1+G_2)=\chi(G_1)+\chi(G_2)$. The same holds even for the indicated chromatic number.
\begin{thm}
$\mathrm{(\cite{pan})}$. \label{join}Let $G = G_1 + G_2$. If $G_1$ is $k_1$-indicated colorable for every $k_1 \geq\chi_i(G_1)$ and $G_2$ is
$k_2$-indicated colorable for every $k_2\geq\chi_i(G_2)$, then $\chi_i(G) = \chi_i(G_1)+\chi_i(G_2)$ and $G$ is $k$-indicated
colorable for all $k\geq\chi_i(G)$.
\end{thm}

Let us recall the structural characterization of $\{P_2\cup P_3, C_4\}$-free graphs, $\{P_5,C_4\}$-free graphs and $\{P_5, (\overline{P_2\cup P_3}),\overline{P_5},\mathrm{Dart}\}$-free graphs which contains an induced $C_5$. The graphs $(\overline{P_2\cup P_3})$ and {Dart} are shown in Figure \ref{free}.

\begin{thm}$\mathrm{(\cite{cho})}$ \label{p2free}
If $G$ is a connected $\{P_2\cup P_3,C_4\}$-free graph, then $G$ is chordal or there exists a partition $(V_1,V_2,V_3)$ of $V(G)$
such that (1) $\langle V_1\rangle\cong \overline{K_m}$, for some $m\geq0$, (2) $\langle V_2\rangle\cong K_t$, for some $t\geq0$, (3) $\langle V_3\rangle$ is isomorphic to a graph
obtained from one of the basic graphs $G_t\ (1\leq t\leq17)$ shown in Figure \ref{basic} by expanding each vertex indicated in circle by a
complete graph (of order $\geq1$), (4) $[V_1,V_3]=\emptyset$ and (5) $[V_2,V_3\backslash S]$ is complete.
\end{thm}

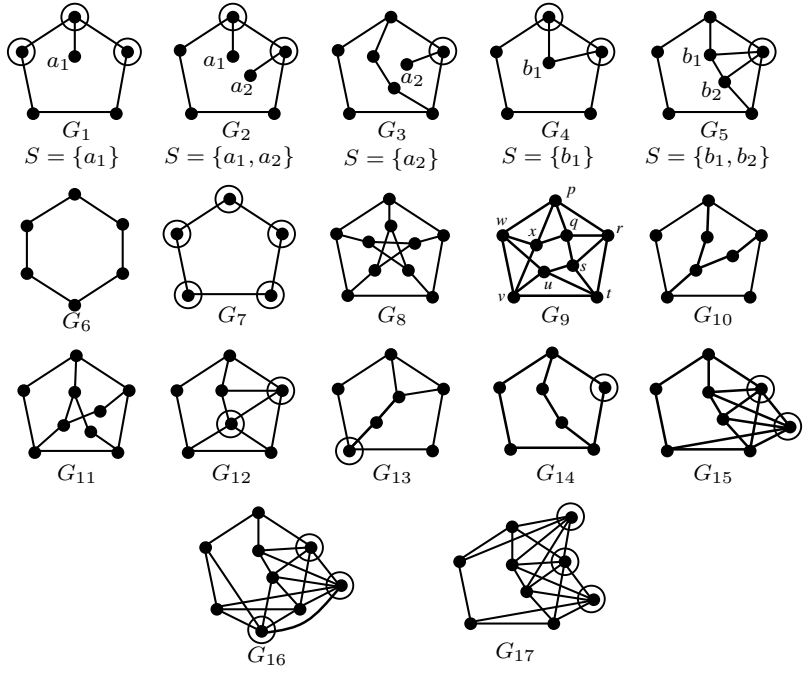
\begin{figure}[t]
  \centering
\scalebox{1.05} 
{
\begin{pspicture}(0,-4.23125)(9.98,4.19125)
\psdots[dotsize=0.16](8.28,2.79125)
\psdots[dotsize=0.16](9.32,2.79125)
\psdots[dotsize=0.16](8.8,4.01125)
\psdots[dotsize=0.16](9.46,3.57125)
\psdots[dotsize=0.16](8.14,3.57125)
\psline[linewidth=0.026cm](8.28,2.81125)(8.12,3.59125)
\psline[linewidth=0.026cm](8.28,2.81125)(9.38,2.81125)
\psline[linewidth=0.026cm](9.46,3.59125)(9.32,2.81125)
\psline[linewidth=0.026cm](8.78,4.03125)(9.48,3.57125)
\psline[linewidth=0.026cm](8.8,4.03125)(8.16,3.61125)
\psdots[dotsize=0.16](8.8,3.53125)
\psdots[dotsize=0.16](8.98,3.19125)
\psline[linewidth=0.026cm](8.8,4.01125)(8.8,3.57125)
\psline[linewidth=0.026cm](8.78,3.55125)(8.98,3.21125)
\psline[linewidth=0.026cm](8.98,3.19125)(9.32,2.81125)
\psline[linewidth=0.026cm](8.96,3.21125)(9.46,3.57125)
\psline[linewidth=0.026cm](8.78,3.53125)(9.42,3.57125)

\psdots[dotsize=0.16](0.32,2.79125)
\psdots[dotsize=0.16](1.36,2.79125)
\psdots[dotsize=0.16](0.84,4.01125)
\psdots[dotsize=0.16](1.5,3.57125)
\psdots[dotsize=0.16](0.18,3.57125)
\psline[linewidth=0.026cm](0.32,2.81125)(0.16,3.59125)
\psline[linewidth=0.026cm](0.32,2.81125)(1.42,2.81125)
\psline[linewidth=0.026cm](1.5,3.59125)(1.36,2.81125)
\psline[linewidth=0.026cm](0.82,4.03125)(1.52,3.57125)
\psline[linewidth=0.026cm](0.84,4.03125)(0.2,3.61125)
\psdots[dotsize=0.16](0.84,3.51125)
\psline[linewidth=0.026cm](0.84,4.05125)(0.84,3.51125)

\psdots[dotsize=0.16](2.3,2.79125)
\psdots[dotsize=0.16](3.34,2.79125)
\psdots[dotsize=0.16](2.82,4.01125)
\psdots[dotsize=0.16](3.48,3.57125)
\psline[linewidth=0.026cm](2.3,2.81125)(2.14,3.59125)
\psline[linewidth=0.026cm](2.3,2.81125)(3.4,2.81125)
\psline[linewidth=0.026cm](3.48,3.59125)(3.34,2.81125)
\psline[linewidth=0.026cm](2.8,4.03125)(3.5,3.57125)
\psline[linewidth=0.026cm](2.82,4.03125)(2.18,3.61125)
\psdots[dotsize=0.16](2.82,3.51125)
\psline[linewidth=0.026cm](2.82,4.05125)(2.82,3.51125)
\psdots[dotsize=0.16](2.16,3.59125)
\psdots[dotsize=0.16](3.04,3.27125)
\psline[linewidth=0.026cm](3.46,3.59125)(3.0,3.27125)

\psdots[dotsize=0.16](4.28,2.79125)
\psdots[dotsize=0.16](5.32,2.79125)
\psdots[dotsize=0.16](4.8,4.01125)
\psdots[dotsize=0.16](5.46,3.57125)
\psline[linewidth=0.026cm](4.28,2.81125)(4.12,3.59125)
\psline[linewidth=0.026cm](4.28,2.81125)(5.38,2.81125)
\psline[linewidth=0.026cm](5.46,3.59125)(5.32,2.81125)
\psline[linewidth=0.026cm](4.78,4.03125)(5.48,3.57125)
\psline[linewidth=0.026cm](4.8,4.03125)(4.16,3.61125)

\psdots[dotsize=0.16](4.58,3.51125)
\psdots[dotsize=0.16](4.14,3.59125)
\psdots[dotsize=0.16](5.0,3.41125)
\psdots[dotsize=0.16](4.84,3.11125)
\psline[linewidth=0.026cm](4.8,4.03125)(4.58,3.51125)
\psline[linewidth=0.026cm](4.56,3.53125)(4.82,3.11125)
\psline[linewidth=0.026cm](4.82,3.11125)(5.32,2.81125)
\psline[linewidth=0.026cm](4.98,3.41125)(5.46,3.59125)

\psdots[dotsize=0.16](6.26,2.79125)
\psdots[dotsize=0.16](7.3,2.79125)
\psdots[dotsize=0.16](6.78,4.01125)
\psdots[dotsize=0.16](7.44,3.57125)
\psdots[dotsize=0.16](6.12,3.57125)
\psline[linewidth=0.026cm](6.26,2.81125)(6.1,3.59125)
\psline[linewidth=0.026cm](6.26,2.81125)(7.36,2.81125)
\psline[linewidth=0.026cm](7.44,3.59125)(7.3,2.81125)
\psline[linewidth=0.026cm](6.76,4.03125)(7.46,3.57125)
\psline[linewidth=0.026cm](6.78,4.03125)(6.14,3.61125)
\psdots[dotsize=0.16](6.78,3.43125)
\psline[linewidth=0.026cm](6.78,4.03125)(6.78,3.45125)
\psline[linewidth=0.026cm](7.44,3.59125)(6.76,3.43125)

\psdots[dotsize=0.16](0.84,1.77125)
\psdots[dotsize=0.16](0.24,1.37125)
\psdots[dotsize=0.16](0.24,0.77125)
\psdots[dotsize=0.16](0.84,0.37125)
\psdots[dotsize=0.16](1.44,0.77125)
\psdots[dotsize=0.16](1.44,1.39125)
\psline[linewidth=0.026cm](0.24,1.39125)(0.24,0.81125)
\psline[linewidth=0.026cm](1.44,1.41125)(1.44,0.77125)
\psdots[dotsize=0.16](2.26,0.49125)
\psdots[dotsize=0.16](3.3,0.49125)
\psdots[dotsize=0.16](2.78,1.71125)
\psdots[dotsize=0.16](3.44,1.27125)
\psdots[dotsize=0.16](2.12,1.27125)
\psline[linewidth=0.026cm](2.26,0.51125)(2.1,1.29125)
\psline[linewidth=0.026cm](2.26,0.51125)(3.36,0.51125)
\psline[linewidth=0.026cm](3.44,1.29125)(3.3,0.51125)
\psline[linewidth=0.026cm](2.76,1.73125)(3.46,1.27125)
\psline[linewidth=0.026cm](2.78,1.73125)(2.14,1.31125)
\psline[linewidth=0.026cm](0.82,1.79125)(1.42,1.39125)
\psline[linewidth=0.026cm](0.84,1.77125)(0.26,1.41125)
\psline[linewidth=0.026cm](0.24,0.77125)(0.8,0.41125)
\psline[linewidth=0.026cm](1.42,0.77125)(0.84,0.39125)

\psdots[dotsize=0.16](4.26,0.49125)
\psdots[dotsize=0.16](5.32,0.47125)
\psdots[dotsize=0.16](4.78,1.71125)
\psdots[dotsize=0.16](5.46,1.25125)
\psdots[dotsize=0.16](4.12,1.27125)
\psline[linewidth=0.026cm](4.26,0.51125)(4.1,1.29125)
\psline[linewidth=0.026cm](4.28,0.49125)(5.38,0.49125)
\psline[linewidth=0.026cm](5.46,1.27125)(5.32,0.49125)
\psline[linewidth=0.026cm](4.76,1.73125)(5.46,1.27125)
\psline[linewidth=0.026cm](4.78,1.73125)(4.14,1.31125)
\psdots[dotsize=0.16](4.8,1.37125)
\psdots[dotsize=0.16](5.1,1.15125)
\psdots[dotsize=0.16](5.02,0.81125)
\psdots[dotsize=0.16](4.6,0.81125)
\psdots[dotsize=0.16](4.52,1.17125)
\psline[linewidth=0.026cm](5.02,0.81125)(5.32,0.47125)
\psline[linewidth=0.026cm](4.62,0.83125)(4.26,0.51125)
\psline[linewidth=0.026cm](4.12,1.27125)(4.5,1.17125)
\psline[linewidth=0.026cm](4.5,1.17125)(5.02,0.81125)
\psline[linewidth=0.026cm](4.54,1.17125)(5.06,1.15125)
\psline[linewidth=0.026cm](5.1,1.17125)(4.62,0.83125)
\psline[linewidth=0.026cm](4.8,1.39125)(4.62,0.83125)
\psline[linewidth=0.026cm](4.78,1.41125)(5.0,0.83125)

\psdots[dotsize=0.16](6.3,-1.44875)
\psdots[dotsize=0.16](7.34,-1.44875)
\psdots[dotsize=0.16](6.82,-0.22875)
\psdots[dotsize=0.16](7.48,-0.66875)
\psdots[dotsize=0.16](6.16,-0.66875)
\psline[linewidth=0.032cm](6.3,-1.42875)(6.14,-0.64875)
\psline[linewidth=0.032cm](6.3,-1.42875)(7.4,-1.42875)
\psline[linewidth=0.032cm](7.48,-0.64875)(7.34,-1.42875)
\psline[linewidth=0.032cm](6.8,-0.20875)(7.5,-0.66875)
\psline[linewidth=0.032cm](6.82,-0.20875)(6.18,-0.62875)

\psdots[dotsize=0.16](8.26,0.47125)
\psdots[dotsize=0.16](9.3,0.47125)
\psdots[dotsize=0.16](8.78,1.69125)
\psdots[dotsize=0.16](9.44,1.25125)
\psdots[dotsize=0.16](8.12,1.25125)
\psline[linewidth=0.026cm](8.26,0.49125)(8.1,1.27125)
\psline[linewidth=0.026cm](8.26,0.49125)(9.36,0.49125)
\psline[linewidth=0.026cm](9.44,1.27125)(9.3,0.49125)
\psline[linewidth=0.026cm](8.76,1.71125)(9.46,1.25125)
\psline[linewidth=0.026cm](8.78,1.71125)(8.14,1.29125)

\psdots[dotsize=0.16](6.7,-0.68875)
\psdots[dotsize=0.16](6.94,-1.10875)
\psdots[dotsize=0.16](8.62,0.81125)
\psdots[dotsize=0.16](8.76,1.23125)
\psdots[dotsize=0.16](9.08,0.99125)
\psline[linewidth=0.032cm](6.84,-0.20875)(6.7,-0.66875)
\psline[linewidth=0.032cm](6.68,-0.68875)(6.92,-1.08875)
\psline[linewidth=0.032cm](6.92,-1.10875)(7.36,-1.44875)
\psline[linewidth=0.032cm](8.78,1.73125)(8.76,1.21125)
\psline[linewidth=0.032cm](8.6,0.81125)(8.22,0.47125)
\psline[linewidth=0.032cm](8.76,1.25125)(8.64,0.87125)
\psline[linewidth=0.032cm](9.08,1.01125)(8.64,0.83125)
\psline[linewidth=0.032cm](9.4,1.25125)(9.08,0.99125)

\psdots[dotsize=0.16](0.34,-1.48875)
\psdots[dotsize=0.16](1.38,-1.48875)
\psdots[dotsize=0.16](0.86,-0.26875)
\psdots[dotsize=0.16](1.52,-0.70875)
\psdots[dotsize=0.16](0.2,-0.70875)
\psline[linewidth=0.026cm](0.34,-1.46875)(0.18,-0.68875)
\psline[linewidth=0.026cm](0.34,-1.46875)(1.44,-1.46875)
\psline[linewidth=0.026cm](1.52,-0.68875)(1.38,-1.46875)
\psline[linewidth=0.026cm](0.84,-0.24875)(1.54,-0.70875)
\psline[linewidth=0.026cm](0.86,-0.24875)(0.22,-0.66875)
\psdots[dotsize=0.16](0.7,-1.14875)
\psdots[dotsize=0.16](0.84,-0.72875)
\psdots[dotsize=0.16](1.16,-0.96875)
\psline[linewidth=0.026cm](0.86,-0.22875)(0.84,-0.74875)
\psline[linewidth=0.026cm](0.68,-1.14875)(0.3,-1.48875)
\psline[linewidth=0.026cm](0.84,-0.70875)(0.72,-1.08875)
\psline[linewidth=0.026cm](1.16,-0.94875)(0.72,-1.12875)
\psline[linewidth=0.026cm](1.48,-0.70875)(1.16,-0.96875)
\psdots[dotsize=0.16](1.04,-1.22875)
\psline[linewidth=0.026cm](1.02,-1.20875)(1.42,-1.48875)
\psline[linewidth=0.026cm](0.82,-0.70875)(1.02,-1.22875)

\psdots[dotsize=0.16](2.26,-1.48875)
\psdots[dotsize=0.16](3.3,-1.48875)
\psdots[dotsize=0.16](2.78,-0.26875)
\psdots[dotsize=0.16](3.44,-0.70875)
\psdots[dotsize=0.16](2.12,-0.70875)
\psline[linewidth=0.026cm](2.26,-1.46875)(2.1,-0.68875)
\psline[linewidth=0.026cm](2.26,-1.46875)(3.36,-1.46875)
\psline[linewidth=0.026cm](3.44,-0.68875)(3.3,-1.46875)
\psline[linewidth=0.026cm](2.76,-0.24875)(3.46,-0.70875)
\psline[linewidth=0.026cm](2.78,-0.24875)(2.14,-0.66875)
\psdots[dotsize=0.16](2.8,-1.12875)
\psdots[dotsize=0.16](2.68,-0.70875)
\psline[linewidth=0.026cm](2.78,-0.24875)(2.68,-0.68875)
\psline[linewidth=0.026cm](2.66,-0.66875)(2.78,-1.12875)
\psline[linewidth=0.026cm](2.66,-0.70875)(3.4,-0.70875)
\psline[linewidth=0.026cm](2.82,-1.08875)(3.4,-0.70875)
\psline[linewidth=0.026cm](2.8,-1.12875)(2.24,-1.46875)
\psline[linewidth=0.026cm](2.78,-1.12875)(3.32,-1.48875)

\psdots[dotsize=0.16](4.28,-1.46875)
\psdots[dotsize=0.16](5.32,-1.46875)
\psdots[dotsize=0.16](4.8,-0.24875)
\psdots[dotsize=0.16](5.46,-0.68875)
\psdots[dotsize=0.16](4.14,-0.68875)
\psline[linewidth=0.026cm](4.28,-1.44875)(4.12,-0.66875)
\psline[linewidth=0.026cm](4.28,-1.44875)(5.38,-1.44875)
\psline[linewidth=0.026cm](5.46,-0.66875)(5.32,-1.44875)
\psline[linewidth=0.026cm](4.78,-0.22875)(5.48,-0.68875)
\psline[linewidth=0.026cm](4.8,-0.22875)(4.16,-0.64875)
\psdots[dotsize=0.16](4.62,-1.10875)
\psdots[dotsize=0.16](4.9,-0.78875)
\psline[linewidth=0.026cm](5.42,-0.68875)(4.92,-0.76875)
\psline[linewidth=0.026cm](4.8,-0.22875)(4.88,-0.78875)

\psdots[dotsize=0.16](8.26,-1.46875)
\psdots[dotsize=0.16](9.3,-1.46875)
\psdots[dotsize=0.16](8.78,-0.24875)
\psdots[dotsize=0.16](9.44,-0.68875)
\psdots[dotsize=0.16](8.12,-0.68875)
\psline[linewidth=0.032cm](8.26,-1.44875)(8.1,-0.66875)
\psline[linewidth=0.032cm](8.26,-1.44875)(9.36,-1.44875)
\psline[linewidth=0.032cm](9.44,-0.66875)(9.3,-1.44875)
\psline[linewidth=0.032cm](8.76,-0.22875)(9.46,-0.68875)
\psline[linewidth=0.032cm](8.78,-0.22875)(8.14,-0.64875)

\psdots[dotsize=0.16](8.78,-0.72875)
\psdots[dotsize=0.16](8.96,-1.06875)
\psline[linewidth=0.032cm](8.78,-0.24875)(8.78,-0.68875)
\psline[linewidth=0.032cm](8.76,-0.70875)(8.96,-1.04875)
\psline[linewidth=0.032cm](8.96,-1.06875)(9.3,-1.44875)
\psline[linewidth=0.032cm](8.94,-1.04875)(9.44,-0.68875)
\psline[linewidth=0.032cm](8.76,-0.72875)(9.4,-0.68875)
\psdots[dotsize=0.16](9.8,-1.16875)
\psline[linewidth=0.032cm](9.42,-0.68875)(9.82,-1.18875)
\psline[linewidth=0.032cm](8.72,-0.70875)(9.8,-1.16875)
\psline[linewidth=0.032cm](8.92,-1.04875)(9.82,-1.18875)
\psline[linewidth=0.032cm](9.28,-1.44875)(9.82,-1.16875)
\psline[linewidth=0.032cm](8.26,-1.44875)(9.84,-1.16875)

\psline[linewidth=0.032cm](4.6,-1.10875)(4.28,-1.44875)
\psline[linewidth=0.032cm](4.9,-0.76875)(4.64,-1.06875)
\psdots[dotsize=0.16](6.34,0.47125)
\psdots[dotsize=0.16](7.38,0.47125)
\psdots[dotsize=0.16](6.86,1.69125)
\psdots[dotsize=0.16](7.52,1.25125)
\psdots[dotsize=0.16](6.2,1.25125)
\psline[linewidth=0.032cm](6.34,0.49125)(6.18,1.27125)
\psline[linewidth=0.032cm](6.34,0.49125)(7.44,0.49125)
\psline[linewidth=0.032cm](7.52,1.27125)(7.38,0.49125)
\psline[linewidth=0.032cm](6.84,1.71125)(7.54,1.25125)
\psline[linewidth=0.032cm](6.86,1.71125)(6.22,1.29125)
\psdots[dotsize=0.16](7.0,1.25125)
\psdots[dotsize=0.16](7.08,0.87125)
\psdots[dotsize=0.16](6.62,1.13125)
\psdots[dotsize=0.16](6.72,0.79125)
\psline[linewidth=0.032cm](6.84,1.73125)(7.0,1.29125)
\psline[linewidth=0.032cm](6.86,1.71125)(6.64,1.17125)
\psline[linewidth=0.032cm](6.16,1.27125)(6.6,1.13125)
\psline[linewidth=0.032cm](6.34,0.49125)(6.62,1.15125)
\psline[linewidth=0.032cm](6.34,0.51125)(6.7,0.79125)
\psline[linewidth=0.032cm](6.7,0.81125)(7.42,0.47125)
\psline[linewidth=0.032cm](7.06,0.89125)(7.38,0.47125)
\psline[linewidth=0.032cm](7.08,0.89125)(6.76,0.79125)
\psline[linewidth=0.032cm](7.0,1.27125)(7.08,0.89125)
\psline[linewidth=0.032cm](7.02,1.27125)(6.6,1.13125)
\psline[linewidth=0.032cm](6.98,1.25125)(7.48,1.25125)
\psline[linewidth=0.032cm](7.1,0.89125)(7.5,1.25125)
\psline[linewidth=0.032cm](6.2,1.25125)(6.72,0.79125)

\pscircle[linewidth=0.024,dimen=outer](1.5,3.57125){0.18}
\pscircle[linewidth=0.024,dimen=outer](2.82,4.01125){0.18}
\pscircle[linewidth=0.024,dimen=outer](0.84,4.01125){0.18}
\pscircle[linewidth=0.024,dimen=outer](0.17,3.57125){0.18}
\pscircle[linewidth=0.024,dimen=outer](3.46,3.58125){0.18}
\pscircle[linewidth=0.024,dimen=outer](5.45,3.581){0.18}
\pscircle[linewidth=0.024,dimen=outer](6.78,4.01125){0.18}
\pscircle[linewidth=0.024,dimen=outer](7.43,3.57125){0.18}
\pscircle[linewidth=0.024,dimen=outer](9.45,3.57125){0.18}

\usefont{T1}{ptm}{m}{it}
\rput(0.8720312,2.56625){\scriptsize $G_1$}
\usefont{T1}{ptm}{m}{it}
\rput(2.8720313,2.56625){\scriptsize $G_2$}
\usefont{T1}{ptm}{m}{it}
\rput(4.8220313,2.56625){\scriptsize $G_3$}
\usefont{T1}{ptm}{m}{it}
\rput(6.8720313,2.56625){\scriptsize $G_4$}
\usefont{T1}{ptm}{m}{it}
\rput(8.862031,2.56625){\scriptsize $G_5$}
\usefont{T1}{ptm}{m}{it}
\rput(0.8720312,0.15625){\scriptsize $G_6$}
\usefont{T1}{ptm}{m}{it}
\rput(2.8220313,0.23625){\scriptsize $G_7$}
\usefont{T1}{ptm}{m}{it}
\rput(4.8220313,0.23625){\scriptsize $G_8$}
\usefont{T1}{ptm}{m}{it}
\rput(6.8720313,0.23625){\scriptsize $G_9$}
\usefont{T1}{ptm}{m}{it}
\rput(8.842031,0.23625){\scriptsize $G_{10}$}
\usefont{T1}{ptm}{m}{it}
\rput(0.87203125,-1.76375){\scriptsize $G_{11}$}
\usefont{T1}{ptm}{m}{it}
\rput(2.8120314,-1.76375){\scriptsize $G_{12}$}
\usefont{T1}{ptm}{m}{it}
\rput(4.8220313,-1.76375){\scriptsize $G_{13}$}
\usefont{T1}{ptm}{m}{it}
\rput(6.8720313,-1.76375){\scriptsize $G_{14}$}
\usefont{T1}{ptm}{m}{it}
\rput(8.842031,-1.76375){\scriptsize $G_{15}$}

\psdots[dotsize=0.16](2.62,-3.46875)
\psdots[dotsize=0.16](3.66,-3.46875)
\psdots[dotsize=0.16](3.14,-2.24875)
\psdots[dotsize=0.16](3.8,-2.68875)
\psdots[dotsize=0.16](2.48,-2.68875)
\psline[linewidth=0.026cm](2.62,-3.44875)(2.46,-2.66875)
\psline[linewidth=0.026cm](3.8,-2.66875)(3.66,-3.44875)
\psline[linewidth=0.026cm](3.12,-2.22875)(3.82,-2.68875)
\psline[linewidth=0.026cm](3.14,-2.22875)(2.5,-2.64875)
\psdots[dotsize=0.16](3.14,-2.72875)
\psdots[dotsize=0.16](3.32,-3.06875)
\psline[linewidth=0.026cm](3.14,-2.24875)(3.14,-2.68875)
\psline[linewidth=0.026cm](3.12,-2.70875)(3.32,-3.04875)
\psline[linewidth=0.026cm](3.32,-3.06875)(3.66,-3.44875)
\psline[linewidth=0.026cm](3.3,-3.04875)(3.8,-2.68875)
\psline[linewidth=0.026cm](3.12,-2.72875)(3.76,-2.68875)
\psdots[dotsize=0.16](4.18,-3.16875)
\psline[linewidth=0.026cm](3.08,-2.70875)(4.16,-3.16875)
\psline[linewidth=0.026cm](3.28,-3.04875)(4.18,-3.18875)
\psline[linewidth=0.026cm](3.64,-3.44875)(4.18,-3.16875)
\psline[linewidth=0.026cm](2.56,-3.44875)(4.2,-3.16875)
\psdots[dotsize=0.16](3.18,-3.74875)

\psline[linewidth=0.026cm](2.44,-2.62875)(3.2,-3.76875)
\psline[linewidth=0.026cm](2.58,-3.44875)(3.18,-3.74875)
\psline[linewidth=0.026cm](3.32,-3.02875)(3.18,-3.74875)
\psline[linewidth=0.026cm](3.66,-3.44875)(3.2,-3.72875)
\psbezier[linewidth=0.032](3.18,-3.76875)(3.72,-3.70875)(3.8,-3.62875)(4.18,-3.16875)
\psdots[dotsize=0.16](5.8,-3.64875)
\psdots[dotsize=0.16](6.84,-3.64875)
\psdots[dotsize=0.16](6.32,-2.42875)
\psdots[dotsize=0.16](6.98,-2.86875)
\psdots[dotsize=0.16](5.66,-2.86875)
\psline[linewidth=0.026cm](5.8,-3.62875)(5.64,-2.84875)
\psline[linewidth=0.026cm](6.98,-2.84875)(6.84,-3.62875)
\psline[linewidth=0.026cm](6.3,-2.40875)(7.0,-2.86875)
\psline[linewidth=0.026cm](6.32,-2.40875)(5.68,-2.82875)
\psdots[dotsize=0.16](6.32,-2.90875)
\psdots[dotsize=0.16](6.5,-3.24875)
\psline[linewidth=0.026cm](6.32,-2.42875)(6.32,-2.86875)
\psline[linewidth=0.026cm](6.3,-2.88875)(6.5,-3.22875)
\psline[linewidth=0.026cm](6.5,-3.24875)(6.84,-3.62875)
\psline[linewidth=0.026cm](6.48,-3.22875)(6.98,-2.86875)
\psline[linewidth=0.026cm](6.3,-2.90875)(6.94,-2.86875)
\psdots[dotsize=0.16](7.34,-3.34875)
\psline[linewidth=0.026cm](6.26,-2.88875)(7.34,-3.34875)
\psline[linewidth=0.026cm](6.82,-3.62875)(7.36,-3.34875)
\psline[linewidth=0.026cm](5.8,-3.62875)(7.38,-3.34875)
\psdots[dotsize=0.16](7.06,-2.30875)
\psline[linewidth=0.026cm](6.48,-3.22875)(7.06,-2.28875)
\psline[linewidth=0.026cm](6.96,-2.84875)(7.06,-2.26875)
\psline[linewidth=0.026cm](6.48,-3.22875)(7.38,-3.36875)
\psline[linewidth=0.026cm](6.3,-2.90875)(7.06,-2.28875)
\psline[linewidth=0.026cm](5.66,-2.86875)(7.06,-2.28875)
\psline[linewidth=0.026cm](6.32,-2.42875)(7.06,-2.28875)

\usefont{T1}{ptm}{m}{it}
\rput(3.2520313,-4.05375){\scriptsize $G_{16}$}
\usefont{T1}{ptm}{m}{n}
\rput(6.3620313,-4.00375){\scriptsize $G_{17}$}

\psline[linewidth=0.026cm](2.62,-3.46875)(3.68,-3.46875)
\psline[linewidth=0.026cm](5.74,-3.64875)(6.9,-3.64875)
\psline[linewidth=0.026cm](4.78,1.73125)(4.78,1.39125)
\psline[linewidth=0.026cm](5.08,1.15125)(5.44,1.25125)
\usefont{T1}{ptm}{m}{it}
\rput(2.6023438,3.43625){\scriptsize $a_1$}
\usefont{T1}{ptm}{m}{it}
\rput(0.64234376,3.41625){\scriptsize $a_1$}
\usefont{T1}{ptm}{m}{it}
\rput(5.0723436,3.21625){\scriptsize $a_2$}
\usefont{T1}{ptm}{m}{it}
\rput(2.9423437,3.10625){\scriptsize $a_2$}
\usefont{T1}{ptm}{m}{it}
\rput(6.581875,3.35625){\scriptsize $b_1$}
\usefont{T1}{ptm}{m}{it}
\rput(8.821875,3.07625){\scriptsize $b_2$}
\usefont{T1}{ptm}{m}{it}
\rput(8.581875,3.51625){\scriptsize $b_1$}

\pscircle[linewidth=0.024,dimen=outer](2.77,1.71125){0.18}
\pscircle[linewidth=0.024,dimen=outer](3.43,1.27125){0.18}
\pscircle[linewidth=0.024,dimen=outer](3.29,0.50125){0.18}
\pscircle[linewidth=0.024,dimen=outer](2.26,0.50125){0.18}
\pscircle[linewidth=0.024,dimen=outer](2.12,1.27125){0.18}

\pscircle[linewidth=0.024,dimen=outer](3.42,-0.70875){0.18}
\pscircle[linewidth=0.024,dimen=outer](2.79,-1.12875){0.18}
\pscircle[linewidth=0.024,dimen=outer](4.28,-1.47){0.18}
\pscircle[linewidth=0.024,dimen=outer](7.47,-0.67){0.18}
\pscircle[linewidth=0.024,dimen=outer](9.43,-0.68875){0.18}
\pscircle[linewidth=0.024,dimen=outer](9.77,-1.15875){0.18}
\pscircle[linewidth=0.024,dimen=outer](3.78,-2.68875){0.18}
\pscircle[linewidth=0.024,dimen=outer](4.16,-3.16875){0.18}
\pscircle[linewidth=0.024,dimen=outer](3.18,-3.72875){0.18}
\pscircle[linewidth=0.024,dimen=outer](7.05,-2.29875){0.18}
\pscircle[linewidth=0.024,dimen=outer](6.98,-2.86875){0.18}
\pscircle[linewidth=0.024,dimen=outer](7.31,-3.33875){0.18}
\psline[linewidth=0.026cm](6.96,-2.86875)(7.32,-3.34875)
\psline[linewidth=0.026cm](3.78,-2.68875)(4.14,-3.14875)

\usefont{T1}{ptm}{m}{it}
\rput(0.80796876,2.23625){\scriptsize $S=\{a_1\}$}
\usefont{T1}{ptm}{m}{it}
\rput(2.7779687,2.23625){\scriptsize $S=\{a_1,a_2\}$}
\usefont{T1}{ptm}{m}{it}
\rput(4.807969,2.21625){\scriptsize $S=\{a_2\}$}
\usefont{T1}{ptm}{m}{it}
\rput(8.767969,2.23625){\scriptsize $S=\{b_1,b_2\}$}
\usefont{T1}{ptm}{m}{n}
\usefont{T1}{ptm}{m}{it}
\rput(6.767969,2.23625){\scriptsize $S=\{b_1\}$}
\psline[linewidth=0.026cm](4.6,-1.10875)(4.28,-1.44875)
\psline[linewidth=0.026cm](4.9,-0.76875)(4.64,-1.06875)
\usefont{T1}{ptm}{m}{it}
\rput(7.0642185,1.79625){\tiny p}
\usefont{T1}{ptm}{m}{it}
\rput(7.0898438,1.39625){\tiny q}
\usefont{T1}{ptm}{m}{it}
\rput(7.662344,1.33625){\tiny r}
\usefont{T1}{ptm}{m}{it}
\rput(7.2109377,0.85625){\tiny s}
\usefont{T1}{ptm}{m}{it}
\rput(7.528281,0.51625){\tiny t}
\usefont{T1}{ptm}{m}{it}
\rput(6.7684374,0.62625){\tiny u}
\usefont{T1}{ptm}{m}{it}
\rput(6.18625,0.47625){\tiny v}
\usefont{T1}{ptm}{m}{it}
\rput(6.1909375,1.43625){\tiny w}
\usefont{T1}{ptm}{m}{it}
\rput(6.572031,1.29625){\tiny x}
\end{pspicture}
}
\caption{Basic graphs used in Theorem \ref{p2free} ($S=\emptyset$ for $G_i$, $6\leq i\leq17$)}
\label{basic}
\end{figure}

\begin{thm}(\cite{fou})\label{p5c4}
Let $G$ be a connected $\{P_5,C_4\}$-free graph.
Then $V(G)=V_1 \cup V_2$ such that\\
(i) $\langle V_1\rangle$ is a $P_5$-free graph which is also chordal. \\
(ii) If $V_2 \neq \emptyset $, then $\langle V_2\rangle=A_1\cup A_2\cup\dots\cup A_l$ where each $A_i$ is a $\mathbb{K}[C_5]$, for every $i\in\{1,2,\dots, l\}$ for some $l\geq1$. Also, $\langle N(A_i)\rangle$ is a complete subgraph of $V_1$ and $[A_i, N(A_i)]$ is complete.
\end{thm}

\begin{thm}(\cite{ara})\label{splitexp}
If $G$ is a connected $\{P_5, (\overline{P_2\cup P_3}),\overline{P_5},Dart\}$-free graph that contains an induced $C_5$, then $G$ is either isomorphic to $C_5(S_1,S_2,S_3,S_4,S_5)$ or $C_5(S_1,S_2,S_3,S_4,S_5)+H$, where $S_i's$ are induced split subgraphs of $G$, $H$ is nonempty and $H\sqsubseteq G$.
\end{thm}

We further claim that the subgraph $H$ mentioned in Theorem \ref{splitexp} is complete.
Let $G$ be a connected $\{P_5, (\overline{P_2\cup P_3}),\overline{P_5},Dart\}$-free graph that contains an induced $C_5$. Let $C=\langle \{v_0,v_1,v_2,v_3,v_4\}\rangle\cong C_5\sqsubseteq G$. Suppose there exists two non-adjacent vertices $x$ and $y$ in $H$. Then $\langle \{v_1,v_2,x,y,v_4\}\rangle\cong (\overline{P_2\cup P_3})\sqsubseteq G$, a contradiction.

It can be noted that $\mathbb{K}[C_5]$ is one of the graphs mentioned in Figure \ref{basic} of Theorem \ref{p2free}. Also $\mathbb{K}[C_5]$ is an induced subgraph of the graphs mentioned in Theorem \ref{p5c4} and \ref{splitexp} . So, we shall first consider the indicated coloring for the complete expansion of $C_5$. Even though $\mathbb{K}[C_5]$ looks simple, a technique as done for $\mathbb{K}[C_6]$ (see Proposition \ref{expc6}) doesn't look possible for $\mathbb{K}[C_5]$. Hence we are forced to adopt a laborious  process.

\begin{thm}
\label{comexp} For  $1\leq i\leq 5$, let $m_i$'s be  positive integers. Then the graph $G=$ \linebreak $\mathbb{K}[C_5](m_1,m_2,m_3,m_4,m_5)$ is $k$-indicated colorable for all $k\geq\chi(G)$.
\end{thm}
\begin{proof}
Let the graph $G=\mathbb{K}[C_5](m_1,m_2,m_3,m_4,m_5)$, where $m_i\geq1$ and $V_i=V(K_{m_i})$ for $1\leq i\leq 5$.
Also, let $k$ be a positive integer such that $k\geq\chi(G)$ and let  $\{1,2,\ldots,k\}$ be the set of colors. We shall show that $G$ is $k$-indicated colorable.  Without much difficult, it can be seen that $\alpha(G)=2$.
Suppose $\omega(G)\geq\frac{|V(G)|}{\alpha(G)}=\frac{|V(G)|}{2}$, let Ann start by presenting the vertices of a maximum clique. Without loss of generality, let it be $V_1\cup V_2$.
By our choice of $k$, $k\geq \chi(G)$, and hence  $|V(G)|\leq2\omega(G)\leq 2k$.
Next, let Ann present the vertices of $V_3$ and $V_5$ in any order. Since $V_1\cup V_5$ and $V_2\cup V_3$ induce a clique and $V_1\cup V_2$ is a maximum clique, we see that $|V_5|\leq |V_2|$ and $|V_3|\leq |V_1|$. Thus Ben will have an  available
color for each of the vertex in $V_3$ and $V_5$. Finally, let Ann present  the vertices of $V_4$. As we have already observed, $|V(G)|\leq 2\omega(G)$ and hence $|V_1|+...+|V_5|\leq 2(|V_1|+|V_2|)$ which in turn implies that $|V_3|+|V_4|+|V_5|\leq k$. Thus here also Ben has an available color for the vertices of $V_4$. Hence,   Ann wins the game on $k$ colors.

Now let us consider the case  when $\omega(G)<\frac{|V(G)|}{2}$. We know that, $\frac{|V(G)|}{2}\leq \chi(G)\leq k$. In this case, let Ann first present the vertices of $V_1$. For $2\leq i\leq5$, let $\mathcal{N}_i$ denote the set of all uncolored vertices in $V_i$ and $\mathcal{C}_i$ denote the set of available colors for the vertices in $\mathcal{N}_i$. Let $c(v)$ denote the color given by Ben to the vertex $v$.

The following are some of the observation regarding $\mathcal{N}_i$ and $\mathcal{C}_i$.
\vskip.3cm
\noindent\textbf{Observations 3.6.1}\\
\textsl{(i) Once when the vertices of $V_1$ are colored by Ben, we have the following values.
\begin{equation*}
|\mathcal{C}_i|-|\mathcal{N}_i|=k-|V_1|-|V_i|>0, \text{ for } i\in\{2,5\}
\end{equation*}
\begin{equation*}
 |\mathcal{C}_i|-|\mathcal{N}_i|=k-|V_i|>0, \text{ for } i\in\{3,4\}
\end{equation*}
\begin{equation*}
 |\mathcal{C}_i\cup\mathcal{C}_{i+1}|-|\mathcal{N}_i|-|\mathcal{N}_{i+1}|=k-(|V_i|+|V_{i+1}|)>0, \text{ for } i\in\{2,3,4\}
\end{equation*}
(Note that, since $k>\omega(G)$, all the values given in the above equations are positive. Also note that, we shall use $\{\}^*$ to indicate these values. For instance, $\{|\mathcal{C}_2|-|\mathcal{N}_2|\}^*=k-|V_1|-|V_2|$).\\
(ii) For $2\leq i\leq5$, the sets $\mathcal{N}_i$ and $\mathcal{C}_i$ constantly change during the coloring process.\\
(iii) For $2\leq i\leq5$, if $|\mathcal{C}_i|-|\mathcal{N}_i|\geq0$, then Ben has an available color at that stage for each of the vertex in $V_i$.\\
(iv) On coloring the vertices of $V_i$, the value of $|\mathcal{C}_i|-|\mathcal{N}_i|$
remains unchanged (since both $|\mathcal{C}_i|$ and $|\mathcal{N}_i|$ are reduced by 1).\\
(v) Similarly, on coloring the vertices of $V_i\cup V_{i+1}$, the value of $|\mathcal{C}_i\cup\mathcal{C}_{i+1}|-|\mathcal{N}_i|-|\mathcal{N}_{i+1}|$ remains unchanged.}

\vskip.3cm
Let us now proceed with the other  uncolored vertices, namely, $V_2\cup V_3\cup V_4\cup V_5$. Without loss of generality, let us assume that $|V_3|\geq|V_4|$ and hence $|\mathcal{C}_3|-|\mathcal{N}_3|\leq|\mathcal{C}_4|-|\mathcal{N}_4|$.
Now let Ann present the vertices of $V_3$ until one of the following holds.

\noindent (i). All the vertices of $V_3$ are colored\\
(ii). $|\mathcal{C}_2|-|\mathcal{N}_2|=0$\\
(iii). $|\mathcal{C}_4\cup\mathcal{C}_{5}|-|\mathcal{N}_4|-|\mathcal{N}_{5}|=0$.

\noindent \textbf{Case 1 :} (i) holds.

In this case, we can easily observe that $|\mathcal{C}_4\cup\mathcal{C}_{5}|-|\mathcal{N}_4|-|\mathcal{N}_{5}|\geq0$ and $|\mathcal{C}_i|-|\mathcal{N}_i|\geq0$, $i=2,4,5$. Next, let Ann present the vertices of $V_2$ in any order. Since $|\mathcal{C}_2|-|\mathcal{N}_2|\geq0$, Ben always has an available color for the vertices of $V_2$. Now for presenting the vertices of $V_4\cup V_5$, let Ann follow the following strategy.

Compare $|\mathcal{C}_4|-|\mathcal{N}_4|$ and $|\mathcal{C}_5|-|\mathcal{N}_5|$. Which ever is smaller Ann present an uncolored vertex from that vertex set, namely, $V_4$ or $V_5$. Do this again and again until $\mathcal{N}_4\cup \mathcal{N}_5=\emptyset$.

Note that even after presenting the vertices of $V_2$, $|\mathcal{C}_i|-|\mathcal{N}_i|\geq0$ for $i=4$ and $5$, and $|\mathcal{C}_4\cup\mathcal{C}_5|-|\mathcal{N}_4\cup\mathcal{C}_5|\geq0$. This together with (iv) and (v) of Observation 3.6.1 guarantees that when Ann follows the above strategy for presenting the uncolored vertices in $V_4\cup V_5$, Ben will have an available color for each of these vertices.

\noindent \textbf{Case 2 :} (ii) holds.

For $|\mathcal{C}_2|-|\mathcal{N}_2|=0$, Ben should have colored the vertices of $V_3$ with exactly $\{|\mathcal{C}_2|-|\mathcal{N}_{2}|\}^*=k-|V_1|-|V_2|$ colors which are not given to the vertices of $V_1$. Next, let Ann present the vertices of $V_2$. Since $|\mathcal{C}_2|-|\mathcal{N}_2|=0$, by using (iii) of Observation 3.6.1, Ben has an available color for each of the vertex in $V_2$. One can easily observe that every time a vertex is colored in $V_2$, the value of $|\mathcal{C}_3|-|\mathcal{N}_3|$ is
reduced by 1. Also $|V_2|+|V_3|\leq \omega(G)<k$ and none of the vertex in $V_4$ is colored. Hence $|\mathcal{C}_3|-|\mathcal{N}_3|\geq0$. Now, let Ann present all the uncolored vertices of $V_3$. Again by (iii) of Observation 3.6.1, Ben has an available color for each of the vertex presented.
A similar argument shows that $|\mathcal{C}_4|-|\mathcal{N}_4|\geq0$ and $|\mathcal{C}_5|-|\mathcal{N}_5|\geq0$. As observed earlier, Ben must have colored the vertices of $V_3$ with exactly $\{|\mathcal{C}_2|-|\mathcal{N}_{2}|\}^*=k-|V_1|-|V_2|$ colors which are not given to the vertices of $V_1$. Hence, we see that
$|\{c(v):v\in V_3\}\cap\{c(v):v\in V_1\}|=|V_3|-\{|\mathcal{C}_2|-|\mathcal{N}_2|\}^*=|V_3|-(k-|V_1|-|V_2|)$. Also observe that, the value of
$|\mathcal{C}_4\cup\mathcal{C}_{5}|-|\mathcal{N}_4|-|\mathcal{N}_{5}|$ reduces by one every time Ben colors a vertex of $V_3$ with a color given to a vertex in $V_1$. Now,\\ $|\mathcal{C}_4\cup\mathcal{C}_{5}|-|\mathcal{N}_4|-|\mathcal{N}_{5}|=
 \{|\mathcal{C}_4\cup\mathcal{C}_{5}|-|\mathcal{N}_4|-|\mathcal{N}_{5}|\}^*-|\{c(v):v\in V_3\}\cap\{c(v):v\in V_1\}|  $\\
 \phantom{$|\mathcal{C}_4\cup\mathcal{C}_{5}|-|\mathcal{N}_4|-|\mathcal{N}_{5}|\ $}$ = (k-|V_4|-|V_5|) - ( |V_3|-(k-|V_1|-|V_2|))$\\
 \phantom{$|\mathcal{C}_4\cup\mathcal{C}_{5}|-|\mathcal{N}_4|-|\mathcal{N}_{5}|\ $}$= 2k - |V(G)|\geq0$.

Thus $|\mathcal{C}_4\cup\mathcal{C}_{5}|-|\mathcal{N}_4|-
|\mathcal{N}_{5}|\geq0$, $|\mathcal{C}_4|-|\mathcal{N}_4|\geq0$ and $|\mathcal{C}_5|-|\mathcal{N}_5|\geq0$. Hence Ann can follow the same strategy as given in Case 1 for presenting the vertices in $V_4\cup V_5$ to yield a winning strategy.

\noindent \textbf{Case 3 :} (iii) holds.\\
For $|\mathcal{C}_4\cup\mathcal{C}_{5}|-|\mathcal{N}_4|-|\mathcal{N}_{5}|=0$, Ben should have colored the vertices of $V_3$ with exactly $ \{|\mathcal{C}_4\cup\mathcal{C}_{5}|-|\mathcal{N}_4|-|\mathcal{N}_{5}|\}^* = k-|V_4|-|V_5|$ colors which are given to the vertices of $V_1$.
As observed in Case 2, $|\mathcal{C}_4|-|\mathcal{N}_4|\geq0$ and $|\mathcal{C}_5|-|\mathcal{N}_5|\geq0$.
Thus Ann follow the same strategy as given in Case 1 for presenting the vertices in $V_4\cup V_5$ to yield a winning strategy. Next for $i=2,3$, $|\mathcal{C}_i|-|\mathcal{N}_i|\geq0$. Since $\{c(v):v\in V_3\}\cap\{c(v):v\in V_4\}=\emptyset$ and $|\{c(v):v\in V_1\}\cap\{c(v):v\in V_3\}|=k-|V_4|-|V_5|$,  we have $|\{c(v):v\in V_4\}\cap\{c(v):v\in V_1\}|=|V_1|-\{|\mathcal{C}_4\cup\mathcal{C}_{5}|-|\mathcal{N}_4|-
  |\mathcal{N}_{5}|\}^*=|V_1|-(k-|V_4|-|V_5|)$. As observed in Case 2, the value of $|\mathcal{C}_2\cup\mathcal{C}_{3}|-|\mathcal{N}_2|-|\mathcal{N}_{3}|$
 reduces by 1 every time Ben colors a vertex in $V_4$  with a color given to a vertex in $V_1$. Now,\\
$|\mathcal{C}_2\cup\mathcal{C}_{3}|-|\mathcal{N}_2|-|\mathcal{N}_{3}| = \{|\mathcal{C}_2\cup\mathcal{C}_{3}|-|\mathcal{N}_2|-|\mathcal{N}_{3}|\}^*-|\{c(v):v\in V_4\}\cap\{c(v):v\in V_1\}|$\\
\phantom{$|\mathcal{C}_2\cup\mathcal{C}_{3}|-|\mathcal{N}_2|-|\mathcal{N}_{3}|$} $= k-|V_2|-|V_3|-(|V_1|-(k-|V_4|-|V_5|))$\\
\phantom{$|\mathcal{C}_2\cup\mathcal{C}_{3}|-|\mathcal{N}_2|-|\mathcal{N}_{3}|$} $= 2k-|V(G)|\geq0.$

\noindent Thus again Ann can follow the same strategy as given in Case 1 for presenting the uncolored vertices in $V_2\cup V_3$ to yield a winning strategy.

\noindent Hence $G$ is $k$-indicated colorable for all $k\geq\chi(G)$.
\end{proof}
\noindent Corollaries \ref{cor2},  \ref{p5c4ind}, \ref{split}, \ref{p5hd} and Theorem \ref{p2p3indi} are some of the consequences of Theorem \ref{comexp}.
\begin{cor}\label{cor2}
For $1\leq i\leq 5$, let $m_i$'s be positive integers. Then for the graph $G=$\linebreak $\mathbb{K}[C_5](m_1,m_2,m_3,m_4,m_5)$, $\chi(G)=\max\left\{\omega(G),\left\lceil\frac{|V(G)|}{2}\right\rceil\right\}$.
\end{cor}
\begin{proof}
We know that, $\chi(G)\geq\max\left\{\left\lceil\frac{|V(G)|}{2}\right\rceil,\omega(G)\right\}$. If one closely observes Theorem \ref{comexp}, it can be seen that $G$ is $k$-indicated colorable for $k=\max\left\{\omega(G),\left\lceil\frac{|V(G)|}{2}\right\rceil\right\}$. Thus $\chi(G)\leq\chi_i(G)\leq\max\left\{\omega(G),\left\lceil\frac{|V(G)|}{2}\right\rceil\right\}$.
\end{proof}

An immediate consequence of Corollary \ref{cor2} is one of the results proved by J. L. Fouquet et.al.  in \cite{fou}. Namely, for  $G=\mathbb{K}[C_5](m,m,m,m,m)$, $m\geq1$, $\chi(G)=
\left\lceil\frac{5m}{2}\right\rceil=\left\lceil\frac{|V(G)|}{2}\right\rceil$.

%
%

\begin{cor}\label{p5c4ind}
If $G$ is a $\{P_5,C_4\}$-free graph, then $G$ is $k$-indicated colorable for all $k\geq\chi(G)$.
\end{cor}
\begin{proof}
By Theorem \ref{union}, it is enough to prove the result for a connected $\{P_5,C_4\}$-free graph.
Let $G$ be such a graph. Then by Theorem \ref{p5c4}, $V(G)=V_1 \cup V_2$ such that\\
(i) $\langle V_1\rangle$ is a $P_5$-free graph which is also chordal. \\
(ii) If $V_2 \neq \emptyset $, then $\langle V_2\rangle=A_1\cup A_2\cup\dots\cup A_l$ where each $A_i$ is a $\mathbb{K}[C_5]$, for every $i\in\{1,2,\dots, l\}$ for some $l\geq1$. Also, $\langle N(A_i)\rangle$ is a complete subgraph of $V_1$ and $[A_i, N(A_i)]$ is complete.

Let the color set be $\{1,2,\ldots,k \geq \chi(G)\}$.
Since $\langle V_1\rangle$ is chordal, col$(\langle V_1\rangle)=\omega(\langle V_1\rangle)\leq k$. By Theorem \ref{col}, Ann has a winning strategy for the vertices of $V_1$ using $k$ colors. Let Ann follow this winning strategy for presenting the vertices of $V_1$. Let $Q_i=\langle N(A_i)\rangle\subseteq V_1$, for every $i\in\{1,2,\ldots,l\}$. Clearly $\chi(A_i+Q_i)=\chi(A_i)+\chi(Q_i)\leq \chi(G)\leq k$. Thus $k-\chi(Q_i)\geq\chi(A_i)$, for $1\leq i\leq l$. By Theorem \ref{comexp}, each $A_i$ is $k$-indicated colorable for every $k\geq \chi(A_i)$ and hence Ann has a winning strategy for each $A_i$ while using $k-\chi(Q_i)\geq\chi(A_i)$ colors, for $1\leq i\leq l$. Since $A_i$'s are disjoint, if  Ann presents the vertices of $A_i$'s for $1\leq i\leq l$ by using these winning strategies, Ben cannot create a blocked vertex. Thus $G$ is $k$-indicated colorable for all $k\geq\chi(G)$.
\end{proof}

Recall that a graph $G$ is said to be a split graph, if $V(G)$ can be partitioned in to two subsets such that the subgraph induced by one set is a clique and the other is totally disconnected.
\begin{cor}\label{split}
Let $S_1,S_2,S_3,S_4,S_5$ be the split graphs. The graph $G=C_5(S_1,S_2,S_3,S_4,S_5)$ is $k$-indicated colorable for all $k\geq \chi(G)$.
\end{cor}
\begin{proof}
Let $G$ be the graph $C_5(S_1,S_2,S_3,S_4,S_5)$. For $1\leq i\leq5$,  $V(S_i)=V_i\cup U_i$, where $\langle V_i\rangle$ is a maximum clique and $\langle U_i\rangle$ is an independent set respectively.
Let the color set be $\{1,2,\ldots,k\geq\chi(G)\}$. By Theorem \ref{comexp}, there is a winning strategy for the subgraph $C_5(V_1,V_2,V_3,V_4,V_5)$ of $G$ using $k$ colors. Let Ann  follow this winning strategy to present the vertices of $C_5(V_1,V_2,V_3,V_4,V_5)$. Next, let Ann presents the remaining vertices of $G$, namely the vertices in $\cup_{i=1}^5 U_i$,  in any order.
Since each of the vertex  $x\in U_i$, $1\leq i\leq5$, has a non neighbor in $V_i$, the color of that non neighbor in $V_i$ will be available for $x$. Thus Ben cannot create any blocked vertex and hence Ann wins the game on $G$ with $k$ colors.
\end{proof}

An immediate consequence of Theorem \ref{join}, \ref{splitexp} and Corollary \ref{split} is Corollary \ref{p5hd}.

\begin{cor}\label{p5hd}
If $G$ is connected $\{P_5, (\overline{P_2\cup P_3}),\overline{P_5}, {Dart}\}$-free graph that contains an induced $C_5$, then $G$ is $k$-indicated colorable for all $k\geq \chi(G)$.
\end{cor}

Now let us consider the indicated coloring of connected $\{P_2\cup P_3,C_4\}$-free graphs.
\begin{thm}
\label{p2p3indi}If $G$ is a connected $\{P_2\cup P_3,C_4\}$-free graph, then $G$ is $k$-indicated colorable for all $k\geq\chi(G)$.
\end{thm}
\begin{proof}
Let $G$ be a connected $\{P_2\cup P_3,C_4\}$-free graph. By Theorem \ref{free}, if $G$ is chordal then col$(G)=\omega(G)=\chi(G)$.
By Theorem \ref{col},
$G$ is $k$-indicated colorable for all $k\geq\chi(G)$. Suppose $G$ is not chordal, then there exists a partition $(V_1,V_2,V_3)$ of $V(G)$
such that $\langle V_1\rangle\cong \overline{K_m}$ for some $m\geq0$, $\langle V_2\rangle\cong K_t$ for some $t\geq0$ and $\langle V_3\rangle\cong G_i$ for some $i$, $1\leq i\leq17$ (see Figure \ref{basic}). Let us divide the proof into two cases as follows.

\noindent \textbf{Case 1 : $V_2=\emptyset$}

Since $[V_1,V_3]=\emptyset$ and $G$ is connected, $G\cong G_j$ for some $j$, $1\leq j\leq 17$. Hence it is enough to show that for $1\leq j\leq 17$, $G_j$ is $k$-indicated colorable for all $k\geq\chi(G_j)$.
Let us first consider the $G_j$'s when $j\in\{1,2,\ldots,17\}\backslash\{6,7,8,9\}$.
It is not difficult to observe that for these $j\in\{1,2,\ldots,17\}\backslash\{6,7,8,9\}$,  col$(G_j)=\chi(G_j)$. Hence by Theorem \ref{col}, $G_j$ is $k$-indicated colorable for all $k\geq\chi(G_j)$.

For $j\in\{6,7,8,9\}$, it can be seen that $\mathrm{col}(G_j)\neq \chi(G_j)$, so we consider these graphs separately. The graph $G_6\cong C_6$ and hence $k$-indicated colorable for all $k\geq \chi(G_6)$. The graph $G_7\cong \mathbb{K}[C_5](m_1,m_2,m_3,m_4,m_5)$, where each $m_i\geq1$, $1\leq i\leq 5$ and hence by Theorem \ref{comexp}, $G_7$ is $k$-indicated colorable for all $k\geq\chi(G_7)$.
Next the graph $G_8\cong P$, the Petersen graph. In \cite{pan}, it has been showed that the Petersen graph $P$ is $k$-indicated colorable for all $k\geq \chi(P)$. Finally, let us consider the graph $G_9$. It is easy to check that $\chi(G_9)=3$ and  col$(G_9)=5$. By Theorem \ref{col}, it is enough to show that $G_9$ is 3 and 4-indicated colorable. Let us first consider $G_9$ with 3 colors, namely $\{1,2,3\}$. If  Ann presents the vertices of $G_9$ in the order $p,q,r,s,t,u,v,w,x$, then Ben always has an available color for each of the vertices. Now let us consider $G_9$ with 4 colors, namely $\{1,2,3,4\}$. Let Ann start by presenting the vertices $p,q,r$. Without loss of generality, let Ben color these vertices with 1, 2 and 3 respectively. Now let Ann present the vertex $u$. Suppose Ben colors $u$ with 1 or 4, then Ann will presents the remaining vertices in the order $s,t,v,x,w$. Suppose Ben colors $u$ with 2 or 3, then Ann will presents the remaining vertices in the order $w,x,v,t,s$. This guarantees the fact that Ben cannot block any of the vertex. Thus Ann wins the game with 4 colors.

\noindent \textbf{Case 2 : $V_2\neq\emptyset$}

Recall that $\langle V_2\rangle\cong K_t$ and $\langle V_3\rangle\cong G_j$, $1\leq j\leq17$. Since $V_1$ is independent and $[V_1,V_3]=\emptyset$, we can color the vertices of $V_1$ with one of the colors of $V_3$. Thus $\chi(G)=\chi(K_t+G_j\backslash S)=t+\chi(G_j\backslash S)$, for $j\in\{1,2,\ldots,17\}$. Let us first consider the graphs $G$ for which $\langle V_3\rangle \cong G_j$, $j\in\{1,2,3,4,5\}$, the graph with $S\neq\emptyset$.
Without much difficulty one can show that $\mathrm{col}(G)=\chi(G)$. Thus by Theorem \ref{col},
$G$ is $k$-indicated colorable for all $k\geq\chi(G)$.

Next, let us consider the graphs $G$ for which $\langle V_3\rangle \cong G_j$, $j\in\{6,7,\ldots, 17\}$. We know that, $K_t$ is $k_1$-indicated colorable for all $k_1\geq t$ and by Case 1, $G_j$ is $k_2$-indicated colorable for all $k_2\geq \chi(G_j)$, $j\in\{6,7,\ldots, 17\}$. Hence by Theorem \ref{join}, we see that $\langle V_2\cup V_3\rangle\cong K_t+G_j$ is $k$-indicated colorable for all $k=k_1 + k_2\geq t+\chi(G_j)=\chi(G)$ and hence Ann has a winning strategy for $\langle V_2\cup V_3\rangle$ while using $k$ colors for any $k\geq \chi(G)$. Next, let Ann present the vertices of $V_1$ in any order. Since $V_1$ is independent and $[V_1,V_3]=\emptyset$, the colors used in $V_3$ are available to Ben for each vertex in $V_1$. Thus $G$ is $k$-indicated colorable for all $k\geq\chi(G)$.
\end{proof}

\subsection*{Acknowledgment} For the first author, this research was supported by the
Council of Scientific and Industrial Research, Government of India, File no: 09/559(0096)/2012-EMR-I.
Also, for the third author, this research was supported by the UGC-Basic Scientific Research, Government of India.

\end{titlepage}
\end{document}